\documentclass[11pt,reqno]{amsart}
\usepackage[top=1in, bottom=1in, left=1in, right=1in]{geometry} 
\usepackage{amsmath, amssymb, graphicx, mathrsfs, hyperref}
\usepackage{amsthm}
\usepackage{palatino,bbm}
\usepackage[makeroom]{cancel}
\usepackage[usenames,dvipsnames,svgnames]{xcolor}
\usepackage{bm}
\usepackage{latexsym,comment}
\usepackage{upref, eucal}
\usepackage[normalem]{ulem} 
\usepackage[all]{xy}
\usepackage{xcolor}
\usepackage{palatino, mathrsfs}
\usepackage[makeroom]{cancel}
\usepackage{url}
\usepackage{hyperref}
\usepackage{enumitem}

\usepackage{tikz,color,soul}
\usetikzlibrary{matrix,arrows}

\usetikzlibrary{calc}

\usepackage{tikz-cd}

\usepackage{bbm}

\newcommand {\nc} {\newcommand} 
\newcommand {\enm} {\ensuremath}

\def \d{\delta}
\def \uuu{u}
\nc {\bdm} {\begin{displaymath}}
	\nc {\edm} {\end{displaymath}}

\numberwithin{equation}{section}
\usepackage{hyperref} 
\definecolor{labelkey}{rgb}{1,0,0}

\newtheorem{theorem}{Theorem}[section]
\theoremstyle{plain} 
\theoremstyle{plain} 
\newcommand{\thistheoremname}{}
\newtheorem{genericthm}[theorem]{\thistheforemname}

\newtheorem*{genericthm*}{\thistheoremname}
\newenvironment{namedthm*}[1]
{\renewcommand{\thistheoremname}{#1}%
	\begin{genericthm*}}
	{\end{genericthm*}}

\newtheorem{lemma}[theorem]{Lemma}

\newtheorem{corollary}[theorem]{Corollary}
\newtheorem*{definition*}{Definition}

\newtheorem*{conjecture*}{Conjecture}
\newtheorem*{theorem*}{Theorem}
\newtheorem*{remark*}{Remark}
\newtheorem*{lemma*}{Lemma}
\numberwithin{equation}{section}

\DeclareMathOperator{\Li}{Li}

\makeatletter
\let\@@pmod\pmod
\DeclareRobustCommand{\pmod}{\@ifstar\@pmods\@@pmod}
\def\@pmods#1{\mkern4mu({\operator@font mod}\mkern 6mu#1)}
\makeatother

\renewcommand{\Re}{\mathrm{Re}}

\nc {\form}[1] {\enm{\mbox{\underline{for}}}_{#1}}
\nc {\prol}[1] {\enm{\mbox{\underline{prol}}_{{#1}^*}}}

\nc {\stk} {\stackrel}

\nc{\beqar}{\begin{eqnarray*}}
	\nc{\eeqar}{\end{eqnarray*}}

\newcommand{\Pn}[2] {\ensuremath{ {\mathbb{P}}^{#1}_{#2}}}
\nc{\Quot}[3]{\enm{ {\mathfrak{Quot}_{ {#1}/{#2}/{#3}}}}}
\nc{\Hilb}[2]{\enm{ {\mathfrak{Hilb}_{ {#1}/{#2}}}}}

\newcommand{\bb}[1]{\mathbb{#1}}

\nc {\Coh}[4] {\ensuremath{H^{#1}(\Pn{#2}{},{#3}({#4}))}}
\nc {\Ch}[3] {\enm{H^{#1}(X_t,{#2}_t({#3}))}}
\nc {\Qphi}[4]{\enm{ {\mathfrak{Quot}^{~#4}_{ {#1}/{#2}/{#3}}}}}
\nc {\Gra}[4]{\enm{ {\mathfrak{Grass}_{#2}({#3},{#4})}}}
\nc {\HomA}[2]{\enm{\mathrm{Hom}_A{#1}{#2}}}
\nc {\tr}{\mathrm{tr}}

\nc {\C}[2]{\binom{#1}{#2}}

\def \mb{\mbox}

   \def \h{\hat{\ }}

  \def \bX{{\bf X}} \def \bH{{\bf H}}

   \def \k{\kappa}

\newcommand{\F}{{\mathbb{F}}}

\def \R1{R((q))[q']\h}

\def \d{\text{deg}}

\def\xi{\chi}

\usepackage{xcolor}

\newcommand{\oldmarginpar}[1]{}

\nc{\bx}{{\mathbf x}}
\nc{\by}{\bold{y}}
\nc{\bz}{\bold{z}}
\nc{\ba}{\bold{a}}
\nc{\Fp}{\tilde{F}}
\nc{\Rp}{\tilde{R}}
\nc{\mlow}{m_{\mathrm{l}}}
\nc{\mup}{m_{\mathrm{u}}}
\nc{\ord}{\mb{ord }}
\nc{\bXp}{\bX_{\mathrm{prim}}}
\nc{\bPsi}{\bold{\Psi}}
\nc{\mult}{\mathrm{mult}}
\nc{\mbB}{\mathbbm{B}}
\nc{\mfor}[1]{{#1}^{\mathrm{for}}}
\nc{\Hdr}{\bH^1_{\mathrm{dR}}(A)}

\nc{\ale}{\textcolor{red}}
\nc{\emma}{\textcolor{orange}}
\nc{\bluev}{\textcolor{blue}}
\nc{\mP}{\mathfrak{p}}

\nc{\prolarrow}[1]{\stk{(\uuu_{#1},\d_{#1})}{\rightarrow}}
\nc{\tc}{\textcolor}
\nc{\nexp}[1]{\exp_{\d,{#1}}}
\nc{\gr}{\mathrm{Gr}}

\nc{\tJ}[1]{{{\mathcal{J}}^n{#1}}} 

\nc{\Sn}{S^{(n)}}
\nc{\Snr}[1]{S^{(n)}_{#1}}
\nc{\Sp}[1]{{\mathbf{Sp}_{#1}}}
\nc{\Aff}[1]{{\mathbf{Aff}_{#1}}}
\nc{\Sch}[1]{{\mathbf{Sch}_{#1}}}
\nc{\undef}{{\color{red} (Undefined)}}
\nc{\ov}[1]{{\overline{#1}}}

\nc{\Qbar}{\overline{\bb{Q}}}
\nc{\Aut}{\mathrm{Aut}}
\newcommand{\Tr}{\mathrm{Tr}}

\title[{An induction principle for the Bombieri-Vinogradov theorem over $\F_q[t]$}]{An induction principle for the Bombieri-Vinogradov theorem over $\F_q[t]$ and a variant of the
	Titchmarsh divisor problem.}
\author{Sampa Dey}

\address{Department of Mathematics, 
	Indian Institute of Technology Gandhinagar, 
	Gandhinagar, 
	Gujarat 382355, India}

\email{sampa.math@gmail.com}

\author[Aditi Savalia]{Aditi Savalia}

\address{Department of Mathematics, 
	Indian Institute of Technology Gandhinagar, 
	Gandhinagar, 
	Gujarat 382355, India}

\email{aditiben.s@iitgn.ac.in}

\begin{document}

	\begin{abstract} Let $\F_q[t]$ be the polynomial ring over the finite field $\F_{q}$. For arithmetic functions $\psi_{1}, \psi_{2}:\F_{q}[t]\rightarrow\mathbb{C}$, we establish that if a Bombieri-Vinogradov type equidistribution result holds for $\psi_{1}$ and $\psi_{2}$, then it also holds for their Dirichlet convolution $\psi_{1}\ast \psi_{2}$. As an application of this, we resolve a version of the Titchmarsh divisor problem in $\F_{q}[t]$. More precisely, we obtain an asymptotic for the average behaviour of the divisor function over shifted products of two primes in $\F_q[t]$.

	\end{abstract}
	\maketitle

%
%
%

	\section{Introduction}
	\label{sec:intro}
	
	The Bombieri-Vinogradov theorem is one of the most celebrated theorems in analytic number theory, concerned with the equidistribution property of primes in arithmetic progressions. To articulate the theorem precisely, we first set up some notation. Let $a$ and $d$ be coprime integers. Denote
	\begin{equation*}
		\pi (x; d, a):= \# \left\lbrace 
		p \leq x: p \equiv a \pmod* d
		\right\rbrace, 
	\end{equation*}  
	where $p$ represents a prime number. Assuming the Generalized Riemann Hypothesis (GRH) for Dirichlet $L$-functions, one can get that for any $d \leq x$, 
	\[
	\pi (x; d, a) = \frac{1}{\phi(d)} \Li (x) 
	+ O
	\left( 
	x^{1/2} \log (x)
	\right) , 
	\] 
	where $\Li(x)$ is the usual logarithmic integral
	\begin{align*}
		\Li(x):=\int\limits_{2}^{x}\frac{dt}{\log t}.
	\end{align*}
	An unconditional result in this context is the well-known Siegel-Walfisz theorem which asserts that for any $N>0$, there exists $c(N)>0$ such that, if $d \leq (\log x)^{N}$,
	\[
	\pi (x; d, a) = \frac{1}{\phi(d)} \Li (x) 
	+ O
	\left( 
	x \exp
	\left(-c(N)(\log x)^{1/2} \right)
	\right),  
	\]
uniformly in $d$. In other words, a non-trivial upper bound on the error term
	\[
	E(x; d, a):=
	\pi (x; d, a) - \frac{1}{\phi(d)} \Li (x) 
	\]
	is unconditionally known in the range $d \leq (\log x)^{N}$, for any $N>0$. 
	In 1965, Bombieri and Vinogradov independently proved that for any $A>0$, there exists $B=B(A)>0$ such that
		\begin{equation}
			\sum
			\limits_{d\leq \frac{x^{1/2}}{(\log x)^{B}}} \max_{(a,d)=1}\max_{y\le x} 
			\big | E (y; d, a)\big |
			\ll_{A}\frac{x}{(\log x)^{A}}.\label{BV for f}
		\end{equation}		
 More generally, an arithmetic function $f$ is said to have \textit{level of distribution} $\theta$, if for any $A>0$, there exists $B=B(A)>0$ such that
\begin{equation}
	\sum
	\limits_{d\leq \frac{x^{\theta}}{(\log x)^{B}}} \max_{(a,d)=1}\max_{y\le x} 
	\bigg | 
	\sum\limits_{\substack{n\leq y\\n \equiv a \pmod*{d}}}
	f(n)
	-
	\frac{1}{\phi(d)}
	\sum\limits_{\substack{n\leq y\\ (n,d)=1}}
	f(n)
	\bigg |
	\ll_{A}\frac{x}{(\log x)^{A}}. \label{BV for f1}
\end{equation}	
	
	The Bombieri- Vinogradov theorem asserts that the level of distribution for the prime indicator function is $\theta = \frac{1}{2}$. Further, it was conjectured by Elliott and Halberstam that the bound \eqref{BV for f1} holds for all $0<\theta<1$ for this function. It is worth noting that Bombieri-Vinogradov theorem yields a bound as strong as what would follow from GRH for Dirichlet $L$-functions. Pushing the level of distribution beyond half has been an active area of research, resulting in important contributions due to Fouvry and Iwaniec \cite{FuIw} \cite{FoIwII} , Bombieri, Friedlander and Iwaniec  \cite{primesinAP_BFI}, \cite{primesinAPII_BFI}, \cite{primesinAPIII_BFI}, Zhang \cite{Zhang}, Maynard \cite{maynard2020primes}, \cite{maynard2020primesII}, \cite{maynard2020primesIII}, Granville and Shao  \cite{GAS} and many others.

	In 1976, Y. Motohashi \cite{Mo} proved an interesting induction principle for the Bombieri-Vinogradov theorem. For any two arithmetic functions $f$ and $g$, their multiplicative convolution or \textit{Dirichlet product} is defined as
	\[	f\ast g (n)=\sum\limits_{ab=n} f(a)g(b).\] For an arithmetic function $f$, we consider the following three properties:
	\begin{itemize}
		\item[(a)] $f(n)=O(\tau(n)^{C})$ for some fixed $C>0$.
		\item[(b)] Let $\chi$ be a non-principal character modulo $d$ such that the conductor of $\chi$ is of order $O((\log x)^{D})$, for $D>0$ suitably large. Then 
		\[
		\sum\limits_{n\leq x} f(n) \chi (n) 
		=
		O \left( 
		\frac{x}{(\log x)^{3D}}
		\right). 
		\]
		\item[(c)] The function $f$ satisfies the Bombieri-Vinogradov type equidistribution property, that is, \eqref{BV for f1} holds with $\theta=\frac{1}{2}$. 
	\end{itemize}
	
	
	Motohashi proved that if $f$ and $g$ satisfy properties $(a)$, $(b)$ and $(c)$, then their multiplicative convolution $f \ast g$ also satisfies these three properties.
	%
	%
	%

	Recently, Darbar and Mukhopadhyay \cite{DaMu} generalized Motohashi's result to imaginary quadratic fields. In this paper, we establish an analogous induction principle for equidistribution in arithmetic progressions, in the setting of $\F_{q}[t]$. While it is true that the Riemann hypothesis is known over finite fields, theorems of Bombieri-Vinogradov type are relevant as they give information about equidistribution in arithmetic progressions for  a variety of functions. For instance, our main result allows us to prove equidistribution in arithmetic progressions for almost primes as well as the $k$-fold divisor function.  Pushing the level of distribution beyond half remains an important question in the function field setting as well as demonstrated by recent work  due to Sawin  \cite{sawin_Duke},   \cite[Theorem 1.2]{sawin2021square},  as well as   Sawin  and Shusterman \cite[Theorem 1.7]{sawin2018chowla}.

	We proceed to state our main result below after setting up relevant notation. 
	Let $\F_{q}$ be a finite field of order $q$ and $\F_{q}[t]$ be the polynomial ring defined over $\F_{q}$. We denote the degree of a polynomial $f$ in $\F_{q}[t]$ by $\d(f)$, and define the norm of a polynomial $|f|$ as $q^{\d(f)}$. Throughout the article, we consider $f,g, h$ to be monic polynomials in $\F_q[t]$. Let $\tau_{k}(f)$ denote the $k$-fold divisor function which counts the number of ways to write $f$ as a product of $k$ \textit{monic} polynomials. When $k=2$, the function $\tau_{2}(f)$ is the usual divisor function $\tau(f)$ which counts the number of monic polynomials dividing $f$. 
	
	Let $m$ be a non-constant polynomial in $\F_{q}[t]$. We may denote the ideal $(m)$ by $m$ without mentioning this explicitly if the usage is clear from the context. A multiplicative character  $\chi$ modulo $m$ is a complex valued multiplicative function on $\left(\F_{q}[t]/(m)\right)^\ast$. We extend $\chi$ to $\F_{q}[t]$ by putting $\chi(f)=\chi(\overline{f})$, where $f\equiv \overline{f}\pmod m$, and zero otherwise. A principal character $\chi_0$ mod $m$ is defined by the property that $\chi_0(f)=1$ if $(f,m)=1$ and $0$ otherwise.  A character $\chi$  mod $m$ is called primitive multiplicative character if there is no $d|m$ such that $(d)\neq (m)$ and  $\chi$ is induced by a character mod $d$. For a character $\chi$ modulo $m$, a monic polynomial $d|m$ is called the conductor of $\chi$ if there is no $d'|d$ such that $\chi$ is induced by a primitive character mod $d'$. For more details we refer the reader to \cite[ch. 4]{Ro} and \cite{Jo}.
	
	For arithmetic functions $
	\psi_1,\psi_2: \F_{q}[t]\rightarrow \mathbb{C}$, their multiplicative convolution denoted by $\psi_1\ast \psi_2$ is defined as
	\begin{align*}
	\psi_1\ast \psi_2(f)=\sum_{gh=f}\psi_1(g)\psi_2(h).
	\end{align*} 
	For an arithmetic function $\psi: \F_{q}[t]\rightarrow \mathbb{C}$, consider the following properties:
	\begin{enumerate}
		\item\label{cond:1}(growth condition) $\psi(f)=O(\tau(f)^{C})$
		for some $C>0$. 
		\item\label{cond:2} (Siegel-Walfisz bound)  For a non-principal character $\chi$ with conductor of  degree $\leq D\log N$, we have
		\[
		\sum\limits_{\deg(f) = N}
		\psi(f)\chi(f)
		=
		O\left( 
		\frac{q^N}{ N^{3D}}
		\right),
		\] 
		for $D>0$ sufficiently large.		
		\item\label{cond:3} (Bombieri-Vinogradov type equidistribution property) For $m, \ell \in \F_{q}[t]$ with $m$ monic, and $(m,\ell)=1$, let
		\begin{equation}
		\label{def:error}
		E(M;m, l; \psi)
		:=
		\sum\limits_{       \substack    { f \equiv l \pmod* {m}\\ \deg(f) = M    }      }
		\psi(f)
		-
		\frac{1}{\phi(m)}
		\sum\limits_{  \substack   {  (f,m)=1 \\  \deg(f) = M  }    }
		\psi(f).
		\end{equation}
		Then for any $A>0$, there exists $B=B(A)>0$ such that
		\begin{equation*}
		\sum\limits_{    \deg (m)   
			\leq \frac{N}{2}-B\log N
		}
		\max\limits_{M\leq N}
		\max\limits_{(m, l)=1}
		|E(M;m,l;\psi)|\ll_{A}
		\frac{q^{N}}{N^{A}}.
		\end{equation*}		
	\end{enumerate}

	\begin{theorem}
		\label{mainthm} Fix $q$. Let $\psi_{1}$, $\psi_{2}$ be arithmetic functions on $\F_{q}[t]$. If $\psi_{1}$ and $\psi_{2}$ have the properties \eqref{cond:1}, \eqref{cond:2} and \eqref{cond:3}, then the Dirichlet convolution $\psi_{1}\ast \psi_{2}$ also satisfies \eqref{cond:1}, \eqref{cond:2} and \eqref{cond:3}.
	\end{theorem}

	An important distinction that presents itself in our function-field analogue is that we derive the corresponding version of the large sieve inequality required for our proof.	Classically, it is well-known that the large
	sieve inequality plays a pivotal role in proving results of this type. In the function field case, various versions of the large sieve inequality have been developed, for instance by	Hsu \cite{Hsu-Largesieve} as well as Baier and Singh (\cite{BaSi}, \cite{BaSiErratum}). While versions of the large sieve inequality with additive characters are known in the function field case, for our purpose, we need a large sieve estimate involving  multiplicative characters. A version of this has recently been given by Klurman, Mangerel and Ter\"{a}v\"{a}inen (\cite[Lemma 4.2]{Oleksiy}). As we shall see in Section \ref{sec:Large sieve} (see the remark following Theorem \ref{multiplicative bound}), the estimate in \cite{Oleksiy} does not suffice for us  and we require an upper bound which is better \textit{on average} over Dirichlet characters $\chi$. We proceed to prove this in Section \ref{sec:Large sieve} (see Theorem \ref{large sieve inequality}).

	Finally, as an application, we also obtain an asymptotic for the average behaviour of the number of divisors over shifts of products of two primes in $\F_q[t]$. It is worth noting that this requires us to invoke a function field analogue of the Brun-Titchmarsh inequality, proved by Hsu in \cite{Hsu-Largesieve}. It is conceivable that the method should	extend to shifts of products of $k$-primes, where $k$ is fixed, though we do not
	do this here.
	

%
	
A direct consequence is the following corollary which we obtain upon using Theorem \ref{mainthm} iteratively.
	
	\begin{corollary}\label{induction result}
		Let $\psi_{i}$, for $i=1, 2,...,n$, be arithmetic functions on $\F_{q}[t]$ such that each of them satisfies properties \eqref{cond:1}, \eqref{cond:2} and \eqref{cond:3}. Then the Dirichlet convolution $\psi_{1}\ast \psi_{2}\ast...\ast \psi_{n}$ also does so.
	\end{corollary} 

	Motohashi's generalization is significant in terms of yielding a family of arithmetic functions for which equidistribution results now become available. In the setting of $\F_{q}[t]$ as well, we find such interesting applications. 
	
Let $\pi(N)$ denote the number of monic irreducible polynomials of degree $N$ in $\F_q[t]$. The prime number theorem (cf. \cite{Ro}, Theorem 2.2) in $\F_q[t]$ gives
\begin{align}
	\pi(N)=\frac{q^N}{N}+O\left(\frac{q^{N/2}}{N}\right).\label{PNT}
\end{align}
We also have the prime number theorem for arithmetic progressions (cf. \cite{Ro}, Theorem 4.8) stated as follows. Let $d, a \in \F_{q}[t]$, $d$ has positive degree and $(a,d)=1$. Then the number of monic irreducible polynomials of degree $N$ in $\F_{q}[t]$ in the arithmetic progression $a\pmod* d$ is given by
\begin{align}
	\pi(N;d,a)=\frac{q^N}{\phi(d)N}+O\left(\frac{q^{N/2}}{N}\right).\label{PNTforAP}
\end{align}
Thus, the  primes are equidistributed in arithmetic progressions with a level of distribution $\theta =\frac{1}{2}$. 
 A natural question that arises is about the level of distribution of products of two primes in arithmetic progressions. Let $\mathbbm{1}_{\mathcal P}$ denote the prime indicator function. Applying Theorem \ref{mainthm} on $\mathbbm{1}_{\mathcal P}$, we have the following result for the indicator function of the product of two primes. 	\begin{corollary} \label{prime indicator function}
		Over the polynomial ring $\F_{q}[t]$, the Dirichlet convolution $\mathbbm{1}_{\mathcal{P}} \ast \mathbbm{1}_{\mathcal{P}}$ satisfies all the three properties \eqref{cond:1}, \eqref{cond:2} and \eqref{cond:3}.
	\end{corollary} 	

Now, taking $\psi\equiv 1$, it is easy to see that $\psi$ has level of distribution $1$ (see Section \ref{sec:result for tau_k}), and satisfies properties \eqref{cond:1}, \eqref{cond:2}. So, our Theorem \ref{mainthm} gives that $\tau_k$ has Bombieri-Vinogradov type inequality, for each $k$. We state this as the following corollary.
\begin{corollary}\label{result for tau_k}
	For each $k\in \mathbb{N}$, the $k$-fold divisor function $\tau_k$ satisfies properties \ref{cond:1}, \ref{cond:2} and \ref{cond:3}.
\end{corollary}

	One of the significant applications of the classical Bombieri-Vinogradov theorem is to the celebrated Titchmarsh divisor problem. In 1930, Titchmarsh \cite{Ti} proved that
	\[
	\sum\limits_{p \leq x}
	\tau(p-a) 
	\sim
C_1x
	\]
	for a fixed integer $a$ and some constant $C_1>1$ under the assumption of GRH. It was only after more than three decades that an unconditional proof of this result was obtained by Linnik \cite{Li-book}.   In \cite{primesinAP_BFI},  Bombieri, Friedlander and Iwaniec  proved a version of the theorem with arbitrary $\log x$ savings. The error term in the dispersion method was further improved by Drappeau  \cite{Drappeau} to obtain power savings under GRH. Several variants of the above sum have been studied by Rodriques \cite{Rod}, Halberstam \cite{Ha}, Fouvry \cite{Fo}, Akbary and Ghioca \cite{AkGh}, Felix \cite{Fe}, Vatwani and Wong \cite{VaPe} and others. Using the Bombieri-Vinogradov theorem, one can show the following  (cf. \cite{RmAc}, Theorem $9.3.1$). For a fixed $a$, there exists a positive constant $c$ such that
		\[
		\sum\limits_{p \leq x}
		\tau(p-a) 
		=
		cx
		+
		O
		\left( 
		\frac{x\log \log x}{\log x}
		\right). 
		\]   
	
	In \cite{Mo}, Motohashi generalized this divisor problem to products of $k$-primes. In the same year, Fuji \cite{Fu} proved that 
	\begin{equation} \label{fuji}
		\sum\limits_{p_{1}p_{2}\leq x}
		\tau(p_{1}p_{2}-1)
		=
		2\frac{\zeta(2)\zeta(3)}{\zeta(6)}		
		x\log\log x 
		+
		O(x),
		\end{equation}	
		where $\zeta(s)$ denotes the Riemann zeta function. 
Drappeau and Topacogullari \cite{Comident} studied analogous sums over integers with a fixed number of distinct prime divisors. More precisely, letting $N\geq 1$ and $\epsilon > 0$, they showed that there exists a constant  $\delta > 0$ and polynomials $P_{h,\ell}^k (X)$ of degree $k-1$ such that, for $1 \leq k \ll \log \log x$ and $|h|\leq x^\delta$,
		\begin{align*}
		\sum_{\substack{|h|<n\leq x\\\omega(n)=k}}\tau(n-h)=x \sum_{ 0\leq \ell \leq N}\frac{P_{h.\ell}^k(\log \log x)}{(\log x)^\ell}+ O\left(\frac{x(\log \log x)^k}{k!(\log x)^{N+1-\epsilon}}\right). 
		\end{align*}
		Here $\omega(n)$ denotes the number of distinct prime divisors of an integer $n$ and the implicit constants depend only on $N$ and $\epsilon$.  
Taking $k=2$ in their result gives a close analogue of \eqref{fuji}. 
Darbar and Mukhopadhyay extended \eqref{fuji}  to imaginary quadratic number fields (cf. \cite{DaMu}, Theorem 1.6). Analogues of this for function fields have been studied extensively in the literature. Let $P$ denote a monic irreducible polynomial in $\F_{q}[t]$ and let $a$ be a fixed non-zero polynomial in $\F_{q}[t]$. Then for a fixed $q$, as $N\rightarrow\infty$, Hsu \cite{Hsu-BrunTitchmarsh} proved that
	\begin{equation*}
		\sum\limits_{\substack{\deg (P)=N}}
		\tau(P-a)
		=\prod \limits_{P|a} 
		\left( 
		1-\frac{1}{|P|}
		\right) 
		\left( 
		1+ \frac{1}{|P|(|P|-1)}
		\right) ^{-1}\frac{\zeta_q(2)\zeta_q(3)}{\zeta_q(6)}
		q^{N}
		+
		O
		\left( 
		\frac{q^{N}\log N}{N}
		\right),
	\end{equation*} 
	where the implied constant depends only on $a$. Here $\zeta_q(s)$ is the zeta function over $\F_{q}[t]$, defined by  \begin{align}\zeta_q(s)=\prod_{P}\left(1-\frac{1}{|P|^s}\right)^{-1}\label{def:zeta}
	   \end{align} 
	for $\Re(s)>1$, where the product runs over all the monic irreducible polynomials in $\F_{q}[t]$.

If we keep  $N$ fixed and let $q \rightarrow \infty$, then Andrade, Bary-Soroker and Rudnick \cite{Ru} obtained the  asymptotic formula 
	\begin{equation*}
		\sum\limits_{\substack{\deg (P)=N}}
		\tau(P-a)
		=
		q^{N}
		+
		\frac{q^{N}}{N}
		+
		O_N
		\left( 
		q^{N-\frac{1}{2}}
		\right).
	\end{equation*}
 
As an application of the induction principle in the setting of $\F_q[t]$, we can obtain a generalization of Hsu's result  to products of $m$-primes where $m$ is fixed. In particular,  we obtain the following analogue of the Titchmarsh divisor problem over $\F_{q}[t]$. 
	
	\begin{theorem}
		\label{Titchmarsh problem} Fix $q$.
		Let $a$ be a fixed non zero polynomial and $P_{i}$ denote a monic irreducible polynomial in $\F_{q}[t]$. Then as $N\rightarrow\infty$, we have
		\begin{equation*}
			\sum\limits_{\deg(P_{1}P_{2})=N} 
			\tau(P_{1}P_{2}-a)
			=2C_a\frac{\zeta_q(2)\zeta_q(3)}{\zeta_q(6)}q^{N}\log N
			+
			O\left( 
			q^{N}
			\log\log N
			\right),
		\end{equation*}
		where $\zeta_q(s)$ is the zeta function over $\F_{q}[t]$ and the constant $C_a$ is given by the product
		\begin{align}
		C_a=	\prod\limits_{P|a} \left( 
			1-\frac{1}{|P|}
			\right) 
			\left( 
			1+\frac{1}{|P|(|P|-1)}
			\right)^{-1}. \label{def: Ca}
		\end{align}
	
	\end{theorem}
	In case we allow both $N$ and $q$ to go to infinity, we get the following version of the above result.
	\begin{theorem}\label{q>inf version}
		Let $a$ be a fixed non zero polynomial and $P_{i}$ denote a monic irreducible polynomial in $\F_{q}[t]$. Then as $N, q\rightarrow\infty$, we have
		\begin{equation*}
			\sum\limits_{\deg(P_{1}P_{2})=N} 
			\tau(P_{1}P_{2}-a)
			=2q^{N}(\log N+\gamma) + O\left(\frac{q^N}{N}(\log N)^2\right)+O\left(q^{N-\frac{1}{2}}\log N\right).
		\end{equation*}
	\end{theorem}
Moreover, letting $q$ be fixed and $N \to \infty$, it is possible  to  extend the above formulas for products of $m$ primes to get
\begin{align*}
	\sum_{\deg (P_1 P_2... P_m)=N}\tau(P_1P_2...P_m-a)\sim c(m) q^N (\log N)^{m-1},
\end{align*}
where $c(m)$ is a constant independent of $N$. This is a tedious but straightforward modification of the proof of Theorem \ref{Titchmarsh problem} which we leave  to the reader.

	The paper is organized as follows. In Section \ref{sec:Preliminary}, we prove some basic results on the divisor function and state a version of Perron's formula over $\F_{q}[t]$. In Section \ref{sec:Large sieve}, we establish a large sieve inequality for multiplicative characters which will be needed to prove our main theorem. The proofs of Theorem \ref{mainthm} and Corollary \ref{result for tau_k} are  contained in Sections \ref{sec:Main theorem} and \ref{sec:result for tau_k} respectively. By invoking an $\F_{q}[t]$-analogue of the Brun-Titchmarsh inequality, we prove Theorem \ref{Titchmarsh problem} and Theorem \ref{q>inf version} in Section \ref{sec:Titchmarsh problem}.

	\section{Preliminaries}
	\label{sec:Preliminary}
	
	In this section, we state some lemmas which will be useful to prove the main theorems in this paper.
	
	Let $k\geq 2$. For the $k$-fold divisor function, we have (
	\cite[Lemma 2.2]{Ru})
			\begin{align*}
			\sum_{\substack{\d(f)=N\\f\text{-monic}}}\tau_k(f)
			=
			{N+k-1 \choose k-1}
			q^N.
		\end{align*}
	For our purpose we will use the following inequality which is a direct consequence of the above result.
	\begin{lemma}\label{bound for tau} 
		Let $k\geq 2$. Then for the $k$-fold divisor function we have 
	
			\begin{enumerate}
				\item[(i)] 	\begin{align*}\sum_{\substack{\d(f)=N\\f\text{-monic}}}\tau_k(f)\ll N^{k-1}q^N,	\end{align*}
				\item[(ii)]	\begin{align*} \sum_{\d(f)\leq N}\frac{\tau_k(f)}{q^{\d(f)}}\ll N^{k}.\end{align*}	
			\end{enumerate}		
	\end{lemma}
Further, the following lemma allows us to bound $(\tau(f))^{d}$ by $ \tau_{k}(f)$ for suitably large values of $k$.
		\begin{lemma}\label{lem:bound for tau2}
		For any fixed positive integer $d$,  there exists $k=k(d)>0$ such that 
	\begin{equation*}\label{bound for tau2}
				(\tau(f))^{d} \leq \tau_{k}(f),
			\end{equation*}
		
		for all $f \in \F_{q}[t]$. In fact this holds for any $k\geq (d+1)!$.
	\end{lemma}
	
	\begin{proof}
		Since $\tau_{k}$ is multiplicative, it is enough to prove the inequality for powers of monic irreducible polynomials. Let $P$ be a monic irreducible polynomial in $\F_{q}[t]$, and $k \geq (d+1)!$ . We have
		\beqar
		\tau_{k}(P^{a})	&=&	{a+k-1 \choose k-1}\\	& \geq &	\bigg (	\frac{ (a+k-1)(a+k-2)...(a+d+1)    }     { (k-1)!   }	\bigg )	(a+1)^{d}\\
		&=&
		C(a,k) (a+1)^{d} \qquad \text{(say).}	\eeqar
		Note that, $(a+1)^{d}=(\tau(P^{a}))^{d}$. We will be done if we can show that the factor $C(a,k)$ is at least $1$. For a given $k$, it is easy to see that $C(a,k) \ge C(1,k)$ and that 
		\begin{align*}
		C(1,k)	&=\frac{k}{(d+1)!} \geq 	1,
		\end{align*}
		for our choice of $k$.
	\end{proof}
	
	In \cite{Hsu-BrunTitchmarsh}, Hsu proved an asymptotic bound on the sum of $\frac{1}{\phi(f)}$ over monic polynomials. Let $\zeta_q(s)$ be the zeta function over $\F_q[t]$ as in \eqref{def:zeta}. Using the identity
	\[
	\prod\limits_{\substack{P \in \F_{q}[t],\\ \text{monic, irreducible}}}
	\left( 
	1+\frac{1}{|P|(|P|-1)}
	\right) 
	=
	\frac{\zeta_q(2)\zeta_q(3)}{\zeta_q(6)},
	\]
	we record Hsu's result here.
	\begin{lemma}\cite[Lemma 3.1]{Hsu-BrunTitchmarsh} 
		\label{bound for phi} Let $g\in \F_q[t]$ be non-zero. We have
		\begin{align*}
			\sum_{\substack{\d(f)\leq N\\ (f,g)=1}}\frac{1}{\phi(f)}
			=
				C_g		\frac{\zeta_q(2)\zeta_q(3)}{\zeta_q(6)}N
			+
			O(1),
		\end{align*}
	where $C_g$ is as defined in \eqref{def: Ca} and the implicit constant depends only on $g$. 
	\end{lemma}
We will  be using this lemma with $g=1$ or with $g$ being a fixed polynomial throughout this paper.  
	For the sake of convenience of the reader, we remark that there is a minor typo in Lemma 3.1 of \cite{Hsu-BrunTitchmarsh}. In the main term on the right-hand side, the product should be over $p\in I, p\nmid g$, where $I$ denotes the set of monic irreducible polynomials in $\F_q[t]$.

	 We will also use the following $\F_q[t]$ analogue of the classical Brun-Titchmarsh inequality concerning the number of primes in an arithmetic progression, derived by Hsu in \cite{Hsu-Largesieve}.
	\begin{lemma}\cite[Lemma 4.3]{Hsu-Largesieve}\label{Brun-Titchmarsh}
		Let $a$, $b$ be non-zero polynomials in $\F_q[t]$ with a monic,
$(a, b)=1$ and $\deg (a)>\deg (b)\geq 0$. Then for any positive integer $N>\deg (a)$,
		\begin{align*}
			\pi(N;a,b)\leq 2\frac{q^N}{\phi(a)(N-\deg (a)+1)},
		\end{align*}	
where $\pi(N; a, b)$ denotes the number of monic irreducible polynomials $f \in \F_q[t]$ such
that $\deg (f)=N$ and $f\equiv b \pmod* a$.
	\end{lemma}	
	
Next, we derive what can be thought of as some version of Perron's formula over $\F_{q}[t]$. This is derived using  Cauchy's residue theorem and will play a crucial role in proving our main theorem. 	
	\begin{lemma}\label{Perron}
		Let $\sum\limits_{f \, \text{monic}} \frac{a(f)}{|f|^{s}}$ be a Dirichlet series with $a(f)\ll |f|^\epsilon$, for any $\epsilon>0$. Let $N \in \mathbb{R}\setminus\mathbb{Z}$. Then for $\sigma>1$, and any $M \geq N$,
		\begin{align*}
			\sum_{\d(f)\leq N}a(f)=\frac{1}{2\pi i} \int_{\sigma-iT}^{\sigma+iT}\sum_{\d(f)\leq M}\frac{a(f)}{|f|^s}\frac{q^{Ns}}{s} ds+O\left(\frac{q^{\sigma N}}{T}\right).
		\end{align*}
		
	\end{lemma}
	
	\begin{proof}
		Let $\lfloor N\rfloor=N_0$. Note that\begin{align*}
				\sum_{\deg (f)\leq N}a(f)=\sum_{N_0+\frac{1}{2}}a(f).
		\end{align*}
	Hence without loss of generality we assume that $N=N_0+\frac{1}{2}$. We have 
		\begin{align}
			\int_{\sigma-iT}^{\sigma+iT}\sum_{\d(f)\leq M}\frac{a(f)}{|f|^s}\frac{q^{Ns}}{s} ds&= \sum_{\d(f)\leq M}a(f)\int_{\sigma-iT}^{\sigma+iT}\frac{(q^{N}/|f|)^s}{s} ds. \label{eq:perron1}
		\end{align}
	 Writing the above integral as 
	 $\displaystyle
	 	\int_{\sigma-iT}^{\sigma+iT}\frac{x^s}{s} ds$, we have from $(2.7)$ on p. $219$ of Tenenbaum \cite{Tenenbaum},
	 \begin{align*}
	 \frac{1}{2\pi i}	\int_{\sigma-iT}^{\sigma+iT}\frac{(q^{N}/|f|)^s}{s} ds= h\left(\frac{q^N}{|f|}\right)+O\left(\frac{q^{\sigma N}}{|f|^\sigma\big(1+T|\log (q^N/|f|)|\big)}\right),
	 \end{align*}
	 where, 	
	 \begin{align*}
	 	h(x)=\left\{\begin{array}{cc}
	 		1 & \text{ if }x>1\\
	 		1/2 & \text{ if }x=1\\
	 		0 & \text{ if } 0<x<1.
	 	\end{array}\right.
 \end{align*}	
Since $\frac{q^{N_0+\frac{1}{2}}}{|f|}$ is never $1$, we obtain 
\begin{align} 
 	\frac{1}{2\pi i}\sum_{\d(f)\leq M}a(f)\int_{\sigma-iT}^{\sigma+iT}\frac{(q^{N}/|f|)^s}{s} ds
 	&=	\sum_{\d(f)\leq N}a(f)\nonumber\\
 	&+O\left(q^{\sigma N}\sum_{\d(f)\leq M}\frac{|a(f)|}{|f|^\sigma\big(1+T|\log(q^{N}/|f|)\big)}\right).\label{eq:perron2}
 \end{align} 
We now analyze the error term above in more detail. If $|f|\leq q^{N_0}$,
we find that the logarithm in \eqref{eq:perron2} is bounded below as follows.
$$\log \left(\frac{q^N}{|f|}\right)=\log \left(\frac{q^{N_0+\frac{1}{2}}}{|f|}\right)\geq \frac{1}{2}\log q.$$
Similarly if $|f| >q^{N_0}+1$, we can again  see that
$$\log \left(\frac{q^N}{|f|}\right)\geq \frac{1}{2}\log q.$$
Thus, the logarithm term in \eqref{eq:perron2} is always $\gg 1$, yielding
\begin{align*}
	\sum_{\deg f\leq N}a(f)=	\int_{\sigma-iT}^{\sigma+iT}\sum_{\d(f)\leq M}\frac{a(f)}{|f|^s}\frac{q^{Ns}}{s} ds+O\left(q^{\sigma N}\sum_{\deg f\leq M}\frac{|a(f)|}{|f|^\sigma}\left(\frac{1}{1+T}\right)\right),
\end{align*}
upon combining \eqref{eq:perron1} and \eqref{eq:perron2}. As $\sum_f\frac{a(f)}{|f|^s}$ converges absolutely for $\Re(s)=\sigma$, we obtain the desired result.

	\end{proof}

	We end the current section by stating the orthogonality property of multiplicative characters. For more details, the reader may refer to \cite[ch. 4]{Ro}.
	\begin{lemma}
		\label{orthogonality}
		For any two characters $\chi_{1}$ and $\chi_{2}$ modulo $m$, we have
	\begin{enumerate}
		\item[(i)] \begin{equation*}
			\frac{1}{\phi(m)}
			\sum\limits_{g \pmod* m}
			\overline{\chi_{1}}(g) \chi_{2}(g)
			=
			\begin{cases}
				1, \,\text{if} \, \chi_{1}=\chi_{2};\\
				0, \, \text{otherwise}.
			\end{cases}
		\end{equation*}
	\item[(ii)] \begin{equation*}
		\frac{1}{\phi(m)}
		\sum\limits_{\chi \pmod* m}
		\chi(f)
		\overline{\chi(g)} 
		=
		\begin{cases}
			1, \,\text{if} \, f \equiv g \pmod* m;\\
			0, \, \text{otherwise}.
		\end{cases}
	\end{equation*}
	\end{enumerate}	
		where $\overline{\chi}$	is defined by $\overline{\chi_{1}}(h)=\overline{\chi_{1}(h)}$, for all $h\in \F_{q}[t]$ and $f, g \in \F_{q}[t] $ are coprime to $m$.
		
	\end{lemma}
	
	\section{The Large sieve inequality over $\F_{q}[t]$}
	\label{sec:Large sieve}

	The large sieve inequality has proved to be a versatile and powerful tool in number theory. The Bombieri-Vinogradov theorem can be considered as one of the finest applications of the large sieve method. The classical large sieve was first introduced by Linnik \cite{Li} around 1941 in the context of solving Vinogradov's hypothesis related to the size of the least quadratic non-residue modulo a prime. It was further developed by  contributions of
	Bombieri, Davenport, Halberstam, Gallagher, and many others. Analogous results on the large sieve inequality over number fields have been proved by Huxley (\cite{Hu1}, \cite{Hu2}) and Hinz (\cite{Hi2}, \cite{Hi1}). 
	In 1971,  Johnsen \cite{Jo} established an analogue of the large sieve inequality for additive characters, and using similar techniques as introduced by Gallagher in \cite{Ga}, extended the theory to multiplicative characters as well. However, it appears that one requires to impose additional conditions on the set of moduli (see \cite{Jo}, p. 173). More generally, Hsu \cite{Hsu-Largesieve} proved a function field analogue of the large sieve inequality in arbitrary dimension. Recently Baier and Singh (\cite{BaSi}, \cite{BaSi2}) extended the large sieve inequality to square moduli and power moduli. In particular, \cite{BaSi} yields results on the large sieve inequality with additive characters in arbitrary dimension with a restricted set of moduli. For our purpose, we concentrate on the dimension one case (cf. \cite{BaSi}, Corollary 6.5).
	
	
	Continuing with the same notation and definitions as in \cite{Hsu-Largesieve} and \cite{BaSi}, let $\F_{q}(t)$ be the rational function field and $\F_{q}(t)_{\infty}$ be the completion of $\F_{q}(t)$ at the prime at infinity denoted by $\infty$. The absolute norm denoted by $|.|_{\infty}$ is defined as
	\[
	\bigg|~\sum\limits_{i=-\infty}^{n} a_{i}t^{i} \bigg|_{\infty}=q^{n},
	\]
	when $0\neq a_{n} \in \F_{q}$. 
	Over $\F_{q}[t]$, this defines the usual norm of a polynomial. Let
	$
	\Tr: \, \F_{q} \rightarrow \F_{p}
	$ be the usual trace map,
	where $p$ is the characteristic of the field $\F_{q}$. Consider the non-trivial group homomorphism 
	$E: \F_{q} \rightarrow \mathbb{C}^{\ast} $ defined as
	\begin{align*}
		E(x)= 
		\exp 
		\left(   \frac{2 \pi i}{p} \Tr(x)    \right),
	\end{align*}
	and define a map $e: \F_{q}(t)_{\infty} \rightarrow \mathbb{C}^{\ast}$ as
	\begin{align*} e\left(\sum\limits_{i=-\infty}^{n} a_{i}t^{i} \right)= E(a_{-1}).
	\end{align*}
	Using the additivity property of the trace function we see that, $E$ is a nontrivial additive character on $\F_{q}$ and consequently $e$ becomes an additive character on $\F_{q}[t]$. In particular, for some fixed $f\in \F_q[t]$, let $\sigma_f:\, \F_{q}[t] \rightarrow \mathbb{C}^{\ast}$, be the additive character
	\begin{align*}
	\sigma_f(g) =e \bigg(    \frac{g}{f}      \bigg).\label{def sigmaf}
	\end{align*}
With the above notation in mind, we record the following result by Baier and Singh for dimension one. Considering the $n=1$ case of Corollary $6.1$ of \cite{BaSi}, we are able to obtain a version of Corollary $6.5$ of \cite{BaSi} with $q+1$ replaced by $q$ as noted below. This is significant to us since we will keep $q$ fixed and $N\rightarrow\infty$.
	\begin{theorem}
		\label{Baier}
		Let $Q\geq 1$ be any natural number. Then
		\begin{equation}
			\label{additive bound}
			\sum\limits_{\substack{\deg (f) \leq Q\\ f\, \text{monic}}}
			\sum\limits_{\substack{h \pmod* f\\ (h,f)=1}}
			\bigg|
			\sum\limits_{\deg (g) \leq N}
			a_{g}e \bigg(    g\frac{h}{f}      \bigg)
			\bigg|^{2}
			\ll
			\left(q^{N}+q^{2Q}\right)
			\sum\limits_{\deg (g) \leq N}
			|a_{g}|^{2},
		\end{equation}
		where $a_{g} \in \mathbb{C}$ for all $g$.	
	\end{theorem}
	\begin{lemma}\cite[Lemma 3]{Jo}\label{res1}
		Let $\chi$ be a primitive character modulo $f$. Let
		\begin{equation}
			\tau(\overline{\chi}) = 
			\sum\limits_{h \pmod* f}
			\overline{\chi}(h) e\left(\frac{h}{f}\right).\label{def:tau}
		\end{equation}
	 For any $g \in \F_{q}[t]$, we have 
		\begin{align}
			\chi(g)\tau(\overline{\chi})=
			\sum\limits_{\substack{h \pmod* f\\ (h,f)=1}}
			\overline{\chi}(h)  e\left(\frac{gh}{f}\right). \label{eq: chitau}
		\end{align}
	
		\end{lemma}
	\begin{proof}
	Consider the case $(g,f)=1$. Let $\overline{g}\in \F_q[t]$ satisfy $g\overline{g}\equiv 1\pmod*f$. Thus, in this case, we have
	\begin{align*}
			\chi(g)\tau(\overline{\chi}) = 
		\sum\limits_{h \pmod* f}
		\overline{\chi}(h\overline{g})e\left(\frac{h}{f}\right)=	\sum\limits_{m \pmod* f}
		\overline{\chi}(m) e\left(\frac{mg}{f}\right).
	\end{align*}
We now consider $(g,f)=d$; $d>1$. Clearly the left-hand side of \eqref{eq: chitau} is zero. 
 Let $f=f_1d$ and $g=g_1d$. Then, applying the division algorithm on $h$ modulo $f_1$ to write $h=f_1\ell+c$, we have that the right-hand side of \eqref{eq: chitau} is
 \begin{align*}
 	\sum\limits_{\substack{h \pmod* f\\ (h,f)=1}}
 	\overline{\chi}(h)  e\left(\frac{gh}{f}\right)=\sum_{c\pmod*{f_1}}e\left(\frac{g_1c}{f_1}\right)\sum_{\ell\pmod*d}\overline{\chi}(f_1\ell+c).
 \end{align*}
Let $S(c):=\sum_{\ell\pmod*d}\overline{\chi}(f_1\ell+c)$. Note that $S(c+kf_1)=S(c)$, for any $k\in \F_q[t]$. For $a\in \F_q[t]$ with $(a,f)=1$ and $a\equiv 1\pmod*{f_1}$,
\begin{align*}
	\overline{\chi}(a)S(c)&=\overline{\chi}(a)\sum_{\ell\pmod*d}		\overline{\chi}(f_1\ell+c)\\
	&=\sum_{a\ell\pmod*d}\overline{\chi}(f_1a\ell+ac)\\
	&=\sum_{\ell'\pmod*d}\overline{\chi}(f_1\ell'+c)=S(c).
\end{align*}
But, $\overline{\chi}(a)\neq 1$ for all such $a$, as $\chi$ is primitive. Thus, $S(c)=0$. This completes the proof.
	\end{proof}
	
	\begin{lemma}\cite[Lemma 2]{Jo}
		\label{res2}
		Let $\chi$ be a primitive character modulo $f$. Then $|\tau(\chi)|^{2}=|f|$, where $\tau$ is defined as in Lemma \ref{res1}.
	\end{lemma}
\begin{proof}
	Let $g\in \F_q[t]$. We have from \eqref{eq: chitau}
	\begin{align*}
		\left|\chi(g)\tau(\overline{\chi})\right|^2 &=
		\sum\limits_{\substack{h \pmod* f\\ (h,f)=1}}
		\overline{\chi}(h)  e\left(\frac{gh}{f}\right)\sum\limits_{\substack{h' \pmod* f\\ (h',f)=1}}
		\chi(h')  e\left(-\frac{gh'}{f}\right) \\
		&=\sum_{\substack{h,h'\pmod* f\\ (hh',f)=1}}\overline{\chi}(h)\chi(h')  e\left(\frac{g(h-h')}{f}\right).
	\end{align*}
Summing over $g$, we have
\begin{align*}
 \left|\tau(\overline{\chi})\right|^2	\sum_{g\pmod*f}\left|\chi(g)\right|^2 &= \sum_{\substack{h,h'\pmod* f\\ (hh',f)=1}}\overline{\chi}(h)\chi(h')  \sum_{g\pmod*f}e\left(\frac{g(h-h')}{f}\right)\\
	&= |f| \sum_{h\pmod*f}|\chi(h)|^2.
\end{align*}
This proves our claim.
\end{proof}
	Using Theorem \ref{additive bound}, and Lemmas \ref{res1} and \ref{res2}, we now obtain the following function-field analogue of the large sieve inequality with multiplicative characters, which we later make use of.
	
	\begin{theorem}
		\label{multiplicative bound} let $(a_g)$ be a sequence of complex
		numbers and $Q, N \in \mathbb{N}$. Then
			\begin{equation} \label{mult bd}
			\sum\limits_{\substack{\deg (f) \leq Q,\\ f\, \text{monic}}}
			\frac{|f|}{\phi(f)}
			 \sum\limits_{\chi}{\vphantom{\sum\limits_{\chi}}}^\ast
			\bigg|\sum\limits_{\deg (g) \leq N}
			a_{g}\chi (g)\bigg|^{2}
			\ll
			(q^{N}+q^{2Q})
			\sum\limits_{\deg (g) \leq N}
			|a_{g}|^{2},
		\end{equation}
		where $ \sum{\vphantom{\sum\limits_{\chi}}}^\ast$ represents that the sum runs over primitive characters modulo $f$.
%
	\end{theorem}
	\begin{proof}Let us write 
		\begin{equation}
						S(\chi)=
						\sum\limits_{\deg (g) \leq N}
						a_{g}\chi (g).\label{def:S}
		\end{equation}
	Recall the definition \eqref{def:tau} of $\tau$. We have
		\begin{align*} 
			\tau(\overline{\chi}) S(\chi)
			&=
			\tau(\overline{\chi}) 
			\sum\limits_{\deg (g) \leq N.}
			a_{g} \chi(g)\\
			&=	\sum\limits_{\deg (g) \leq N.}
			a_{g} \chi(g)\tau(\overline{\chi}) \\
			&=
			\sum\limits_{\deg (g) \leq N}
			a_{g}
			\sum\limits_{\substack{h \pmod* f\\ (h,f)=1}}
			\overline{\chi}(h) e\bigg( g \frac{h}{f} \bigg),
		\end{align*}
		using \eqref{eq: chitau}. Therefore,
		\begin{align*} 
			\tau(\overline{\chi}) S(\chi)
			&=
			\sum\limits_{\substack{h \pmod* f\\ (h,f)=1}}
			\overline{\chi}(h)
			U(h/f),
		\end{align*}
		where
		\begin{equation*}
			U(x):=\sum\limits_{\deg (g) \leq N}
			a_{g} e\big( g x\big).
		\end{equation*}
		This gives
		\begin{align*} 
			 \sideset{}{^\ast}\sum \limits_{\chi\pmod* f}
			|\tau(\overline{\chi}) S(\chi)|^{2}
			&= \sideset{}{^\ast}\sum \limits_{\chi\pmod* f}
			\bigg|
			\sum\limits_{\substack{h \pmod* f\\ (h,f)=1}}
			\overline{\chi}(h)
			U(h/f)
			\bigg|^{2}\\
			&=
			\sum\limits_{\substack{h \pmod* f\\ (h,f)=1\\(h',f)=1}}
			U(h/f) \overline{U(h'/f)}
			\sum\limits_{\chi\pmod* f}
			\overline{\chi}(h)
			\chi(h').
		\end{align*}
		Using Lemma \ref{res2} for the left-hand side and the orthogonality property of multiplicative characters (Lemma \ref{orthogonality}) on the right-hand side, we have
		\begin{align*} 
			|f|
			 \sideset{}{^\ast}\sum \limits_{\chi\pmod* f}
			|S(\chi)|^{2}
			&\leq
			\phi(f)
			\sum\limits_{\substack{h \pmod* f\\ (h,f)=1}}
			|U(h/f)|^{2}.
		\end{align*}
		Therefore,
		\begin{align*} 
			\frac{|f|}{\phi(f)}
			 \sideset{}{^\ast}\sum \limits_{\chi\pmod* f}
			|S(\chi)|^{2}
			&\leq
			\sum\limits_{\substack{h \pmod* f\\ (h,f)=1}}
			|U(h/f)|^{2}.
		\end{align*}
		Finally, summing over all monic polynomials $f$ having degree $\leq Q$ and using Theorem \ref{Baier},
		we obtain the result. 
	\end{proof}
 Observe  that from Lemma 4.2 of Klurman, Mangerel and Ter\"{a}v\"{a}inen \cite{Oleksiy}, one can obtain the   bound 
\begin{align}
\frac{|f|}{\phi(f)} 	 \sum \limits_{\chi \pmod* f}
\bigg|\sum_{\deg (g) \le  N}a_g\chi(g)\bigg|^{2}  
\le  
 ( q^{N} + 2q^{\text{deg}(f)} )	\sum\limits_{\deg (g) \leq N}
|a_{g}|^{2}. 
\end{align}
Summing this over monic polynomials $f$ of degree $t$ and then letting $t$ run from $D$ to $Q$, we obtain 
\begin{align}
			\sum\limits_{\substack{D< \deg (f) \leq Q\\ f\, \text{monic}}} \frac{|f|}{\phi(f)}  \sum \limits_{\chi \pmod* f}
			\bigg|\sum_{\deg (g) \le  N}a_g\chi(g)\bigg|^{2} 
			 \ll  			
( q^{N+Q} + q^{2Q} )	\sum\limits_{\deg (g) \leq N}
			|a_{g}|^{2}. 
\end{align}
Our bound \eqref{mult bd} gives a savings of $q^Q$ in the first main term by comparison. This is necessary to obtain the  factor of $q^{-D}$ in the   modified form of the large sieve inequality stated in the next theorem.  This plays a  crucial role in  the proof of Theorem \ref{mainthm}.

	\begin{theorem}
		\label{large sieve inequality}let $(a_g)$ be a sequence of complex
		numbers and $D, Q, N \in \mathbb{N}$ . Then
		\begin{equation}
			\sum\limits_{\substack{D< \deg (f) \leq Q\\ f\, \text{monic}}}
			\frac{1}{\phi(f)}
			 \sideset{}{^\ast}\sum \limits_{\chi}
			\bigg|\sum_{\deg (g) \leq N}a_g\chi(g)\bigg|^{2}
			\ll
			(q^{N-D}+q^{Q})
			\sum\limits_{\deg (g) \leq N}
			|a_{g}|^{2},
		\end{equation}
	where $\sideset{}{^\ast}\sum $ represents that the sum
	runs over primitive characters modulo $f$.
	\end{theorem}
	\begin{proof}
		Let $S(\chi)$ be as defined in \eqref{def:S}. Let $\psi : \mathbb{R} \rightarrow \mathbb{R}$ be the positive-valued decreasing continuous function defined as
		$	 \psi(t)={q^{-t}}$.	For any polynomial $f \in \F_{q}[t]$ and $t\in \mathbb{N}$, we define
		\begin{equation*}
			a_{f}:=\frac{|f|}{\phi(f)}
			 \sideset{}{^\ast}\sum \limits_{\chi\pmod* f}
			|S(\chi)|^{2}, \qquad a_{f}(t)=\sum\limits_{\substack{\deg (f)=t,\\f\, \text{monic}}}
			a_{f} \qquad \text{ and } \qquad	A(x):=\sum\limits_{\substack{t\leq x}}
			a_{f}(t).
		\end{equation*}
		Using partial summation we have,
		\begin{equation}\label{partial sum1}
			\sum\limits_{\substack{D<t\leq Q}}
			\psi(t)a_{f}(t)
			=
			\psi(Q)A(Q)-\psi(D)A(D)
			-\int\limits_{D}^{Q}
			A(t)\psi'(t) \, dt.
		\end{equation}		
		Note that the left-hand side of \eqref{partial sum1} is
		\begin{equation}\label{partial sum2}
			\sum\limits_{\substack{D< \deg (f) \leq Q\\f\, \text{monic}}}
			\frac{1}{\phi(f)}
			 \sideset{}{^\ast}\sum \limits_{\chi\pmod* f}
			|S(\chi)|^{2}.
		\end{equation}
		From Theorem \ref{multiplicative bound} we have,
		\begin{equation*}
			A(x)\ll \left(q^{N}+q^{2x}\right)Z,
		\end{equation*}
		where $Z:=\sum\limits_{\deg (g) \leq N}
		|a_{g}|^{2}$.
		Therefore, the right-hand side of equation \eqref{partial sum1} is
		\begin{align}
			&\ll\left( q^{N-D}+q^{Q} \right)Z. \label{partial sum3}
		\end{align}	
		The result thus follows from \eqref{partial sum2} and \eqref{partial sum3}.\end{proof}	
	
	\section{Proof of Theorem \ref{mainthm}}
	\label{sec:Main theorem}
	Let $\psi_1$ and $\psi_2$ satisfy properties \eqref{cond:1}, \eqref{cond:2} and \eqref{cond:3}. Using Lemma \ref{lem:bound for tau2}, we have a constant $k$ such that $\psi_{1}(f)\ll \tau_{k}(f)$ and $\psi_{2}(f)\ll \tau_{k}(f)$. It is easy to derive property \eqref{cond:1} for the convolution $\psi_{1}\ast \psi_{2}$. In the classical case, it is well-known that the Bombieri-Vinogradov type bound implies that of Siegel-Walfisz type. For the sake of completeness, we present the $\F_q[t]$ analogue of this fact. Let us first assume that $\psi_{1}\ast\psi_{2}$ has property \eqref{cond:3} and derive property \eqref{cond:2} from that. 
	 
	Let $\chi$ be a non-principal character modulo $d$ having conductor $d'$ of degree $\leq D\log N$, for sufficiently large $D$. Let $\chi_1$ be a primitive character modulo $d'$ such that $\chi(f)=\chi_{0}(f)\chi_{1}(f)$ for all $f\in \F_q[t]$. Write $d=d'r$. We see that
	\begin{align*}
		\sum\limits_{\deg(f) = N}
		\psi_{1}\ast\psi_{2}(f)\chi(f)&=\sum_{\deg (f)=N}\psi_{1}\ast\psi_{2}(f)\chi_1(f)\chi_0(f)\\
		&=\sum_{\substack{\beta\pmod*{d'}\\ (\beta,r)=1}}\chi_1(\beta)\sum_{\substack{f\equiv \beta \pmod*{d'}\\ \deg (f) =N}}\psi_{1}\ast\psi_{2}(f).
	\end{align*}
	Using property \eqref{cond:3} with $A=4D$ for the inner sum and fundamental property of $\mu$, we derive that the above sum is,
	\begin{align*}
		&\frac{1}{\phi(d')}\sum_{\substack{\deg (f) =N\\ (f,d')=1}}\psi_{1}\ast\psi_{2}(f)\sum_{\substack{\beta\pmod*{d'}}}\chi_1(\beta)\sum_{\substack{k|\beta\\k|r}}\mu (k)+O\left(\frac{q^{N}|d'|}{N^{4D}}\right)\\
		&=\frac{1}{\phi(d')}\sum_{\substack{\deg (f) =N\\ (f,d')=1}}\psi_{1}\ast\psi_{2}(f)\sum_{\substack{k|r\\(k,d')=1}}\mu(k)\sum_{\substack{k\beta'\pmod*{d'}}}\chi_1(k\beta')+O\left(\frac{q^N}{N^{3D}}\right)\\
		&\ll \frac{q^N}{ N^{3D}},
	\end{align*}
	as the sum  $\displaystyle \sum_{k\beta\pmod*{d'}}\chi(k\beta)$ vanishes due to orthogonality property of Dirichlet characters.

	The remainder of this section is devoted to proving property \eqref{cond:3}. The crucial part here is judicious use of the large sieve inequality. Let $m$-monic and $l \in \F_{q}[t]$ with $(m,l)=1$.
	
	 From property \eqref{cond:3} of $\psi_{1}$ and $\psi_{2}$, for any fixed constant $A>0$, we have a constant $B=B(A)>0$ such that
	\begin{align}
		\label{error bound psi1}
		\sum\limits_{       \deg (m)\leq Q       }
		\max\limits_{M\leq N}
		\max\limits_{(m, l)=1}
		|E(M;m,l;\psi_1)|\ll
		\frac{q^{N}}{N^{A}}, 
	\end{align}
	and
	\begin{align}
		\label{error bound psi2}
		\sum\limits_{       \deg (m)\leq Q       }
		\max\limits_{M\leq N}
		\max\limits_{(m, l)=1}
		|E(M;m,l;\psi_2)|\ll
		\frac{q^{N}}{N^{A}}; 
	\end{align}
	where $Q=\frac{N}{2}-B\log N$.
	
As $\psi_{1}(f),\psi_{2}(f)\ll \tau_{k}(f)$, we have $E(M;m,l;\psi_1\ast\psi_2)\ll q^N/N^{B-k+1}$ trivially for $M\leq \frac{N}{2}$. Thus, we only need to deal with $\frac{N}{2}<M\leq N$. Let $\frac{N}{2}<M\leq N$. By the definition in \eqref{def:error}, we have
	\begin{align}
		\label{mainexp}
		E(M;m,l;\psi_{1}\ast\psi_{2})
		=
		\sum\limits_{\substack{gh\equiv l \pmod* m\\ \d(gh)=M}}
		\psi_{1}(g)\psi_{2}(h)
		-\frac{1}{\phi(m)}
		\sum_{\substack{  (gh,m)=1 \\  \d(gh) = M  }}
		\psi_{1}(g)\psi_2(h).
	\end{align}
	We divide the range of summation over $\d(g)$ in both sums in \eqref{mainexp} into three sub-ranges as follows: (i) $\d(g)\leq A'\log N$, (ii)  $A'\log N<\d(g)\leq M-B'\log N$ and (iii) $M-B'\log N<\d(g) \leq M$;
	where $A', B'>0$ are suitable constants to be chosen later.
	Since $(gh, m)=1$, the contribution to \eqref{mainexp} from the range (i) is
	\begin{align*}
		&	\sum_{\substack{ \d(g)\leq A'\log N\\(g,m)=1}}\psi_{1}(g)\sum_{\substack{\d(h)\leq M-\d(g)\\h\equiv g^{-1}l\pmod* m}}\psi_{2}(h)-\frac{1}{\phi(m)}\sum_{\substack{\d(g) \leq A'\log N \\  (g,m)=1 }}\psi_{1}(g)\sum_{\substack{\d(h)\leq M-\d(g)\\(h,m)=1}}\psi_2(h),
	\end{align*}
	which can be re-written as
	\begin{align*}
		\sum_{\substack{ \d(g)\leq A'\log N\\(g,m)=1}}\psi_{1}(g)E(M-\d(g);m,g^{-1}l;\psi_2).
	\end{align*}
	Expressing the two other sums arising from the ranges (ii) and (iii) in the same way, \eqref{mainexp} can be written as
	\begin{align*}
		E(M;m,l;\psi_{1}\ast \psi_{2})&=\sum_{\substack{ \d(g)\leq A'\log N\\(g,m)=1}}\psi_{1}(g)E(M-\d(g);m,g^{-1}l;\psi_2)\\
		&+\sum_{\substack{ A'\log N<\d(g)\leq M-B'\log N \\(g,m)=1}}\psi_{1}(g)E(M-\d(g);m,g^{-1}l;\psi_2)\\
		&+\sum_{\substack{ \d(h)\leq B'\log N\\(h,m)=1}}\psi_{2}(h)E(M-\d(h);m,h^{-1}l;\psi_1)\\
		&:=S_1+S_2+S_3 \qquad \text{(say)}.
	\end{align*}
	On the sum $S_{1}$, using \eqref{error bound psi2} and Lemma \ref{bound for tau}(i), we obtain
	\begin{align*}
		\sum_{\deg (m)\leq Q}\max\limits_{M\leq N}
		\max\limits_{(m, l)=1}S_1 
		&\ll \frac{q^N}{N^A}\sum_{\d(g)\leq A'\log N}\tau_{k}(g) \\
		&\ll \frac{q^N}{N^A}(A'\log N)^{k-1}q^{A'\log N}.
	\end{align*}
	Now, considering the logarithm base $q$, we conclude that
	\begin{align}
		\sum_{\deg (m)\leq Q}\max\limits_{M\leq N}
		\max\limits_{(m, l)=1}S_1 
		& \ll \frac{q^N}{N^{A-A'-\epsilon}},
		\label{bound for S1}
	\end{align}
	for any $\epsilon>0$. Similarly, using \eqref{error bound psi1} and  Lemma \ref{bound for tau}(i) we obtain for any $\epsilon>0$,
	\begin{align}
		\sum_{\deg (m)\leq Q}\max\limits_{M\leq N}
		\max\limits_{(m, l)=1}S_3 
		& \ll \frac{q^N}{N^{A-A'-\epsilon}}.
		\label{bound for S3}
	\end{align}
	Next, we turn our attention to
	\begin{align*}
		S_4:=	\sum_{\deg (m)\leq Q}\max\limits_{M\leq N}
		\max\limits_{(m, l)=1}S_2.
	\end{align*}	
	Choosing $D_1=\frac{A'}{2}\log N$, we split
	the range of $\deg (m)$ and write $S_4$ into two subsums
	\begin{equation}
		S_4 = \sum_{\deg (m)\leq D_1}\max\limits_{M\leq N}
		\max\limits_{(m, l)=1}S_2
		+\sum_{D_1<\deg (m)\leq Q}\max\limits_{M\leq N}
		\max\limits_{(m, l)=1}S_2 \\
		=:T_1+T_2.\label{bound for S4}
	\end{equation}
	Now onwards, we use the notation $A_{1}:=A'\log N$ and $A_{2}:=M-B'\log N$. 	
	Using Lemma \ref{orthogonality}, we express $S_{2}$ as
	\begin{align}
		S_2&
		=\frac{1}{\phi(m)}\sum_{\substack{\chi \pmod* m\\ \chi\neq \chi_0}}\overline{\chi}(l)\sum_{\substack{A_1<\d(g)\leq A_2}}\psi_{1}(g)\chi(g)\sum_{\substack{\d(h)\leq M-\d(g)}}\psi_{2}(h)\chi(h),\label{S2 with chi}
	\end{align}
	where $\chi_0$ denotes the principal character modulo $m$. Putting \eqref{S2 with chi} in $T_1$, we see that
	\begin{align*}
		T_1 \leq \sum_{\deg (m)\leq D_1}\frac{1}{\phi(m)}\max \limits_{M\leq N}\sum_{\substack{\chi \pmod* m\\ \chi\neq \chi_0}} \Bigg|~\sum_{\substack{A_1<\d(g)\leq A_2}}\psi_{1}(g)\chi(g)\sum_{\substack{\d(h)\leq M-\d(g)}}\psi_{2}(h)\chi(h)\Bigg|.
	\end{align*} 
	Now, for some $\delta \in (0,1)$, we can write the inner sum over $g$ and $h$ in \eqref{S2 with chi}, say $T_1'(\chi)$, as
	\begin{align*}
		T_1'(\chi)
		&=\sum_{A_1<\d(g)\leq (1-\delta)M}\psi_{1}(g)\chi(g)\sum_{\d(h)\leq M-\d(g)}\psi_{2}(h)\chi(h)\\
		&\qquad\qquad\qquad\qquad+\sum_{\d(h)\leq \delta M}\psi_{2}(h)\chi(h)\sum_{(1-\delta)M<\d(g)\leq \min{ (M-\d(h), A_2)}}\psi_{1}(g)\chi(g).
	\end{align*}
	Using property \eqref{cond:1} and \eqref{cond:2} for $\psi_{1}$ and $\psi_{2}$ and the bound from Lemma \ref{bound for tau}(ii), we have
	\begin{align}
	T_1'(\chi)	&\ll \frac{q^M}{ M^{3A'/2}}\left(\sum_{A_1<\d(g)\leq (1-\delta)M}\frac{\tau_{k}(g)}{q^{\d(g)}}+\sum_{\d(h)\leq \delta M}\frac{\tau_{k}(h)}{q^{\d(h)}}\right)\ll\frac{q^M}{M^{3A'/2-k}}.\label{bound for T'}
	\end{align}
	 Thus,
		\begin{align}
		\label{bound for T1} T_1\ll \frac{q^N}{N^{3A'/2-k}}\sum_{\deg (m)\leq D_1}1\ll\frac{q^N}{N^{A'-k}} .
	\end{align}	
	Now, to estimate $T_2$, we first write the sum \eqref{S2 with chi} in terms of primitive characters and separate the terms coming from characters having small conductor. Note that, any character $\chi$ modulo $m$ having conductor $m_1$ can be expressed as $\chi(g)=\chi_1(g)\chi_{0}(g)$
	for all $g \in \F_{q}[t]$, where $\chi_{0}$ is the principal character modulo $m$ and $\chi_1$ is a primitive character modulo $m_1$. 
	Let $m=m_1m_{2}$. Using the inequality $\phi(m_1m_{2})\geq \phi(m_1)\phi(m_{2})$, we obtain	
	\begin{align}
		T_2
		\leq 
		&	\nonumber
		\sum\limits_{\deg (m_{2})\leq Q} 
		\max\limits_{M\leq N}
		\max\limits_{(m_{2}, l)=1}
		\frac{1}{\phi(m_{2})}\left(
		 \sum_{\deg (m_1)\leq D_{1}}\max\limits_{(m_1, l)=1}
		 \frac{1}{\phi(m_1)}\Bigg|~ \sideset{}{^\ast}\sum\limits_{\chi \pmod* {m_1}}\overline{\chi}(l)T_1'(\chi)
		 \Bigg|\right.\\
		&\left.+\sum_{D_1<\deg (m_1)\leq Q}\max\limits_{(m_1, l)=1}
		\frac{1}{\phi(m_1)}\Bigg|~ \sideset{}{^\ast}\sum\limits_{\chi \pmod* {m_1}}\overline{\chi}(l)T_1'(\chi)
		\Bigg|\right)\nonumber
		\\
		&=:
		U_{1}+U_{2} \qquad \text{(say)},
		\label{bound for T2}
	\end{align}	
where $ \sum{\vphantom{\sum\limits_{\chi}}}^\ast$ denotes the sum is over primitive characters. The first sum $U_1$ can be bounded using the above upper bound for $T_1'(\chi)$ in \eqref{bound for T'} and Lemma \ref{bound for phi} as follows.
	\begin{align}
	U_{1}
		&\ll 
		\sum_{\deg (m_{2})\leq Q}
		\max\limits_{M\leq N}
		\frac{1}{\phi(m_{2})}
		\sum_{\deg (m_{1})\leq D_{1}}
		\frac{q^M}{M^{3A'/2-k}} \nonumber\\
		&\ll \frac{q^{N}}{N^{A'-k-1}}.	\label{final bound for U1}
	\end{align}
	Next, we define $\widetilde{\psi}_{1}(f)=\psi_{1}(f)$ and $\widetilde{\psi}_{2}(f)=\psi_{2}(f)$ if $(f,m_{2})=1$, and $0$ else. Thus,
	\begin{align*}
		U_2=
		&\sum\limits_{\deg (m_{2})\leq Q} 
		\max\limits_{M\leq N}
		\max\limits_{(m_{2}, l)=1}
		\frac{1}{\phi(m_{2})}
		\sum\limits_{D_1<\deg (m_1)\leq Q}
		\max\limits_{(m_1, l)=1}
		\frac{1}{\phi(m_1)}\\
		&\qquad\qquad\qquad \Bigg|~
		 \sideset{}{^\ast}\sum\limits_{\chi \pmod* {m_1}}
		\overline{\xi}(l)
		\sum\limits_{A_1<\d(g)\leq A_2}
		\widetilde{\psi}_{1}(g)\xi(g)
		\sum\limits_{\d(h)\leq M-\d(g)}
		\widetilde{\psi}_{2}(h)\xi(h)
		\Bigg|. 
	\end{align*}	
	Now, let $M_0=M+\frac{1}{2}$, $\sigma=1+\frac{1}{2N}$ and $T>M$ to be specified later. Note that $\tau_k(f)\ll |f|^\epsilon$, for any $\epsilon>0$. Thus, applying the function field analogue of Perron’s formula given in Lemma \ref{Perron} on $U_{2}$, using Lemma \ref{bound for phi} to simplify the error, we obtain
	\begin{align}
		U_2
		\leq	
		&\sum\limits_{\deg (m_{2})\leq Q}
		\max\limits_{M\leq N}
		\max\limits_{(m_{2}, l)=1}
		\frac{1}{\phi(m_{2})}
		\sum\limits_{D_1<\deg (m_1)\leq Q}
		\max\limits_{(m_1, l)=1}
		\frac{1}{\phi(m_1)}\nonumber\\
		&\qquad\qquad\Bigg| ~\sideset{}{^\ast}\sum\limits_{\chi \pmod* {m_1}}
		\overline{\xi}(l)
		\int_{\sigma-iq^T}^{\sigma+iq^T}
		\widetilde{\Phi}_1(\xi,s)\widetilde{\Phi}_2(\xi,s)\frac{q^{M_0s}}{s} ds
		\Bigg|  
		+ O
		\left( 
		\frac{ q^{\frac{3N}{2}}}    {q^{T}N^{B-1}}
		\right),
		\label{bound for U2}
	\end{align}
	where
	\begin{equation*}
		\widetilde{\Phi}_1(\xi,s):=\sum_{A_1<\d(g)\leq A_2}\frac{\widetilde{\psi}_1(g)\xi(g)}{|g|^{s}}\qquad \text{ and }\qquad
		\widetilde{\Phi}_2(\xi,s):=\sum_{\d(h)\leq T}\frac{\widetilde{\psi}_2(h)\xi(h)}{|h|^{s}}.
	\end{equation*}
	We define
	\begin{align*}
		I_{M,m_2}(l)
		&:=	
		\sum\limits_{D_1<\deg (m_1)\leq Q}
		\max\limits_{(m_1, l)=1}
		\frac{1}{\phi(m_1)}
		\Bigg|~		 \sideset{}{^\ast}\sum\limits_{\chi \pmod* {m_1}}
		\overline{\xi}(l)
		\int_{\sigma-iq^T}^{\sigma+iq^T}
		\widetilde{\Phi}_1(\xi,s)
		\widetilde{\Phi}_2(\xi,s)\frac{q^{M_0s}}{s} ds
		\Bigg| .
	\end{align*}	
	Now, we divide the sum over $\deg (m_{1})$ in the form
	\begin{align*}
		I_{M,m_2}(l)
		=\sum\limits_{j=0}^{J}I_{M,m_2}^{(j)}(l),
	\end{align*}
	where $J+D_1 < Q \leq J+1+D_1$ and 
	\begin{align*}
		I_{M,m_2}^{(j)}(l)
		= \sum\limits_{\deg (m_1)\in R_j}
		\max\limits_{(m_1, l)=1}
		\frac{1}{\phi(m_1)}
		\Bigg|~	 \sideset{}{^\ast}\sum_{\xi \pmod* {m_1}}
		\overline{\xi}(l)
		\int_{\sigma-iq^T}^{\sigma+iq^T}
		\widetilde{\Phi}_1(\xi,s)
		\widetilde{\Phi}_2(\xi,s)\frac{q^{M_0s}}{s} ds
		\Bigg|.
	\end{align*}	
	Here $R_j$ denotes the range $ D_1+j<\deg (m_1) \leq j+1+D_1$ for each $j=0,.., (J-1)$, and $J+D_1<\deg (m_1)\leq Q$ at $j=J$. Also, we divide the integral $I_{M,m_2}$ into two sub parts by writing $\widetilde{\Phi}_2(\xi, s) $ as
	\begin{align*}
		\widetilde{\Phi}_2{(\xi, s)}
		&=\sum\limits_{ h\leq Y_j}
		\frac{\psi_2(h)\xi(h)}{|h|^s}
		+\sum\limits_{Y_j< h\leq T}
		\frac{\psi_2(h)\xi(h)}{|h|^s}\\
		&:=
		\widetilde{\Phi}_{2j}^{(1)}{(\xi, s)}
		+\widetilde{\Phi}_{2j}^{(2)}{(\xi, s)},
	\end{align*}
	where $Y_j<T$ will be chosen later. Now, for each $j$, we have
	\begin{align*}
		&\int_{\sigma-iq^T}^{\sigma+iq^T}
		\widetilde{\Phi}_1(\xi,s)
		\widetilde{\Phi}_{2j}^{(2)}(\xi,s)
		\frac{q^{M_0s}}{s} ds
		-
		\int_{1/2-iq^T}^{1/2+iq^T}
		\widetilde{\Phi}_1(\xi,s)
		\widetilde{\Phi}_{2j}^{(2)}(\xi,s)
		\frac{q^{M_0s}}{s} ds\\
		&=
		\int_{\sigma-iq^T}^{1/2-iq^T}
		\widetilde{\Phi}_1(\xi,s)
		\widetilde{\Phi}_{2j}^{(2)}(\xi,s)\frac{q^{M_0s}}{s} ds
		+
		\int_{1/2+iq^T}^{\sigma+iq^T}
		\widetilde{\Phi}_1(\xi,s)
		\widetilde{\Phi}_{2j}^{(2)}(\xi,s)
		\frac{q^{M_0s}}{s} ds\\
		&\ll
		\frac{q^{M_0}}{q^T}
		\left(
		\sum_{A_1<\gamma\leq A_2}
		q^{-\gamma/2}
		\sum\limits_{\d g =\gamma}
		\tau_{k}(g)
		\right)
		\left(
		\sum_{\gamma \leq Y_j}
		q^{-\gamma/2}
		\sum_{\d h = \gamma}
		\tau_{k}(h)
		\right)\\
		&\ll \frac{q^{3N/2}q^{Y_j/2}(Y_j)^k(A_2)^k}{q^T N^{B'/2}},
	\end{align*}
	using Lemma \ref{bound for tau}(i) and keeping in mind that $A_{2}=M-B' \log N\leq N$. Therefore, for each $j$ we can write
	\begin{align}\label{integration for j}
		I_{M,m_2}^{(j)}(l)
		=\int_{1/2-iq^T}^{1/2+iq^T}S_j^{(1)}(s)
		\frac{q^{M_0s}}{s}ds
		+\int_{\sigma-iq^T}^{\sigma+iq^T}S_j^{(2)}(s)
		\frac{q^{M_0s}}{s}ds
		+O\left(
		\frac{q^{3N/2}q^{Y_j/2}(Y_j)^k(A_2)^k}{q^T N^{B'/2}}
		\right),
	\end{align}
	where
	\begin{align}
		S_j^{(\alpha)}:
		=\sum_{\deg (m_1) \in R_j}
		\max\limits_{(m_1, l)=1}
		\frac{c}{\phi(m_1)}
			 \sideset{}{^\ast}\sum_{\xi \pmod* {m_1}}
		\overline{\xi}(l)
		\widetilde{\Phi}_1(\xi,s)
		\widetilde{\Phi}_{2j}^{(k)}(\xi,s),
	\end{align}
	for $\alpha=1,2 $ and some complex number $c$ with $|c|=1$. Now, using the Cauchy-Schwarz inequality first for the sum over the primitive characters $\chi$ and then for the sum over $m_1$, we obtain
	\begin{align*}
		S_j^{(1)}(s)
		\leq &
		\left(
		\sum_{\deg (m_1) \in R_j}
		\frac{1}{\phi(m_1)}
			 \sideset{}{^\ast}\sum_{\xi \pmod* {m_1}}
		|\widetilde{\Phi}_1(\xi,s)|^2
		\right)^{1/2}\\
		&\qquad\qquad\qquad\qquad	
		\times
		\left(
		\sum_{\deg (m_1) \in R_j}
		\frac{1}{\phi(m_1)}
			 \sideset{}{^\ast}\sum_{\xi \pmod* {m_1}}
		|\widetilde{\Phi}_{2j}^{(1)}(\xi,s)|^2
		\right)^{1/2}.
	\end{align*}	
	Next, applying the large sieve inequality as in Theorem \ref{large sieve inequality} for $s=1/2+it$ such that $-q^{T} \leq t \leq q^{T}$, and Lemma \ref{bound for tau}(ii), we have for $j<J$,
	\begin{align*}
		S_j^{(1)}(s)
		&\ll 
		\left(
		\left(
		q^{D_1+j+1}+q^{A_2-D_1-j}
		\right)
		\sum_{\deg (g) \leq A_2}
		\frac{\tau_{k}(g)}{|g|}
		\right)^{1/2}
		\left(
		\left(
		q^{D_1+j+1}+q^{Y_j-D_1-j}
		\right)
		\sum_{\deg (h) \leq Y_j}\frac{\tau_{k}(h)}{|h|}
		\right)^{1/2}\\
		&\ll 
		\left(
		q^{2Q}+q^{Y_j+1}+q^{A_2+1}+q^{A_2+Y_j-2D_1-2j}
		\right)^{1/2}
		A_{2}^{k/2}Y_{j}^{k/2}.
	\end{align*}
	Similarly, for $j=J$ also, we can get the same upper bound. Choose $Y_j=2D_1+2j$. Thus,
	\begin{align}
		S_j^{(1)}(1/2+it)\ll\frac{q^{N/2}}{N^{\frac{1}{2}\min \{2B,B'\}-k}}.\label{bound for Sj1}
	\end{align}
	Now, we turn to $S_j^{(2)}(s)$ at $s=\sigma +it$, such that $-q^{T} \leq t \leq q^{T}$. We divide the sums $\widetilde{\Phi}_{1}(\xi,s)$ and $\widetilde{\Phi}_{2j}^{(2)}(\xi, s)$ as follows. Define for $A_1+I<A_2\leq A_1+I+1$,
	\begin{align*}
		\widetilde{\Phi}_{1}(\xi,s)=\sum_{i=0}^{I}	\widetilde{\Phi}_{1i}(\xi,s),
	\end{align*}
	where $	\displaystyle\widetilde{\Phi}_{1i}(\xi,s)
	:=\sum_{A_1+i<\deg (g) \leq A_1+i+1}
	\frac{\widetilde{\psi}_1(g)\xi(g)}      {|g|^{s}}$; 
 and for  $Y_j+R<T\leq Y_j+R+1$,
	\begin{align*}
		\widetilde{\Phi}_{2j}^{(2)}(\xi,s)=\sum_{r=0}^{R}	\widetilde{\Phi}_{2jr}^{(2)}(\xi, s),
	\end{align*}
	where $\displaystyle	\widetilde{\Phi}_{2jr}^{(2)}(\xi,s)
	:=
	\sum_{Y_j+r<\deg (h) \leq Y_j+r+1}
	\frac{\widetilde{\psi}_2(g)\xi(g)}{|g|^{s}}$.\\			
	Thus, we can write $S_j^{(2)}(s)$ as
	\begin{align*}
		S_j^{(2)}(s)
		=
		\sum_{i=0}^{I}
		\sum_{r=0}^{R}
		\sum_{\deg (m_1)\in R_j}
		\frac{c}{\phi(m_1)}
			 \sideset{}{^\ast}\sum_{\chi \pmod* {m_1}}
		\overline{\chi_1}(l)
		\widetilde{\Phi}_{1i}(\xi,s)
		\widetilde{\Phi}_{2jr}^{(2)}(\xi,s).
	\end{align*}
	Applying the Cauchy-Schwarz inequality twice as in $S_j^{(1)}(s)$, we obtain that $S_{j}^{(2)}(s)$ is bounded above by
	\begin{align*}
		\sum_{i=0}^{I}\sum_{r=0}^{R}
		\left(\!
		\sum_{\deg (m_1)\in R_j}\!
		\frac{1}{\phi(m_1)}
		 \sideset{}{^\ast}\sum_{\xi \pmod* {m_1}}
		\left|
		\widetilde{\Phi}_{1i}(\xi,s)
		\right|
		\right)^{1/2}
		\left(\!
		\sum_{\deg (m_1)\in R_j}
		\frac{1}{\phi(m_1)}\!
			 \sideset{}{^\ast}\sum_{\xi \pmod* {m_1}}
		\left|
		\widetilde{\Phi}_{2jr}^{(2)}(\xi,s)
		\right|
		\right)^{1/2}\!.
	\end{align*}
Again applying the large sieve inequality from Theorem \ref{large sieve inequality} for $s=\sigma +it$ and Lemma \ref{bound for tau}(ii), we have for $j<J$ and for some $k'>0$,
	\begin{align}
		S_j^{(2)}(s)
		\ll 
		\sum_{i=0}^{I}
		\sum_{r=0}^{R}
		&\left(
		\left(
		q^{(A_1+i+1)-(D_1+j)}+q^{D_1+j+1}
		\right)
		\sum_{A_1+i<\deg(g)\leq A_1+i+1}
		\frac{\tau_{k'}(g)}{|g|^{2}} 
		\right)^{1/2}\nonumber \\
		&
		\left(
		\left(q^{(Y_j+r+1)-(D_1+j)}
		+q^{D_1+j+1}
		\right)
		\sum_{Y_j+r<\deg(h)
			\leq Y_j+r+1}
		\frac{\tau_{k'}(h)}{|g|^{2}} 
		\right)^{1/2}\nonumber\\
		\ll 
		\sum_{i=0}^{I}
		\sum_{r=0}^{R}
		&\left(
		\left(
		q^{(A_1+i+1)-(D_1+j)}
		+q^{D_1+j+1}
		\right)
		\frac{(A_1+i+1)^{k'}}{q^{A_1+i}} 
		\right)^{1/2}\nonumber\\
		&
		\left(
		\left(
		q^{(Y_j+r+1)-(D_1+j)}
		+q^{D_1+j+1}
		\right)
		\frac{(Y_j+r+1)^{k'}}{q^{Y_j+r}} 
		\right)^{1/2}
		\nonumber\\
		\ll 
		\sum_{i=0}^{I}
		\sum_{r=0}^{R}
		&\left(
		q^{-D_1-j}+q^{j-D_1-i}
		\right)^{1/2}
		\left(q^{-D_1-j}
		+q^{-D_1-j-r}
		\right)^{1/2}
		(A_2)^{k'/2}
		(T)^{k'/2},
		\label{bound for Sj2}
	\end{align}
	because of our choices $A_1=2D_1$ and $Y_j=2D_1+2j$. We choose $T=2N$. Since the lengths of the sums over $i$ and $r$ are $O(A_{2})$ and $O(T)$ respectively, and the summand is maximum at $i=0$ and $r=0$, from \eqref{bound for Sj2}, we obtain
	\begin{equation}
		\label{final bound for Sj2}
		S_j^{(2)}(s)
		\ll
		q^{-D_1}N^{k'+2}
		=N^{k'+2-A'/2},
	\end{equation}
	putting the value $D_{1}=\frac{A'}{2}\log N$.
	Therefore, using \eqref{bound for Sj1}  and \eqref{final bound for Sj2} in \eqref{integration for j}, for each $j$ we have
	\begin{align*}
		I_{M,m_2}^{(j)}(l)
		\ll 
		\frac{q^{N}}{N^{\frac{1}{2}\min \{2B,B'\}-k}}+ \frac{q^N}{N^{A'/2-k'-2}},
	\end{align*}
	and hence 
	\begin{align}
		I_{M,m_2}(l)
		\ll
		\frac{q^{N}}{N^{\frac{1}{2}\min \{2B,B'\}-k-1}}+ \frac{q^N}{N^{A'/2-k'-3}},
		\label{final integral}
	\end{align}
	since the length of the sum over $j$ is $\ll Q \ll N$. 
	Using Lemma \ref{bound for phi}, and putting \eqref{final integral} in \eqref{bound for U2}, we obtain 
	\begin{align}
		U_2\ll  \frac{q^{N}}{N^{\frac{1}{2}\min \{2B,B'\}-k-2}}+ \frac{q^N}{N^{A'/2-k'-4}}.
		\label{final bound for U2}
	\end{align}  
	Therefore, using \eqref{final bound for U1} and \eqref{final bound for U2} in \eqref{bound for T2}, we have
	\begin{equation}
		T_2
		\ll
		\frac{q^{N}}{N^{A'-k-1}}
		+
		\frac{q^{N}}{N^{\frac{1}{2}\min \{2B,B'\}-k-2}}+ \frac{q^N}{N^{A'/2-k'-4}}.\label{final bound for T2}
	\end{equation} 
	Combining \eqref{bound for T1} and \eqref{final bound for T2} in \eqref{bound for S4}, we get
	\begin{equation}\label{final bound for S4}
		S_{4} \ll
		\frac{q^N}{N^{A'-k}}
		+
		\frac{q^{N}}{N^{A'-k-1}}
		+
		\frac{q^{N}}{N^{\frac{1}{2}\min \{2B,B'\}-k-2}}+ \frac{q^N}{N^{A'/2-k'-4}}.
	\end{equation}
	Suppose $k''=\max\left\lbrace k,k'\right\rbrace $ and  we choose $A'=B'=A/2$.
	Combining all the bounds \eqref{bound for S1}, \eqref{bound for S3} and \eqref{final bound for S4}, we have for $A>8k''+32$,
	\begin{align*}
		\sum\limits_{\deg (m)\leq Q} \max\limits_{M\leq N}\max\limits_{(m, l)=1}
		|E(M;m,l;\psi_1\ast \psi_2)|\ll
		\frac{q^{N}}{N^{A/8}}.
	\end{align*}
Choosing $B=B(8A)$, we get the required form. Thus we have derived property \eqref{cond:3} for $\psi_{1}\ast\psi_{2}$.

	\section{Proof of Corollary \ref{result for tau_k}} 	\label{sec:result for tau_k}
Consider $\psi\equiv 1$. Fix some $m, \ell \in \F_{q}[t]$ with $(m,\ell)=1$, $\deg (\ell)<\deg (m)$. We have
	\begin{align}
		\sum_{\substack{f\equiv \ell \pmod m\\ \deg (f)=M}}\psi(f)=\sum_{\substack{f=md+\ell \\ \deg (f)=M}}1=\sum_{\deg (d)=M-\deg(m)}1=q^{M-\deg (m)},\label{a*}
	\end{align}
	and 
	\begin{align}
		\sum_{\substack{  (f,m)=1 \\  \deg(f) = M  }}\psi(f) &= \sum_{\deg (f)=M}\sum_{d|f, d|m}\mu(d)\nonumber\\
		&=\sum_{d|m}\mu(d)\sum_{\substack{d|f\\ \deg(f)=M}}1 \nonumber\\
		&=\sum_{d|m}\mu(d)q^{M-\deg(d)}=q^M\prod_{P|m}\left(1-\frac{1}{|P|}\right)=\phi(m)q^{M-\deg (m)}\label{b*}
	\end{align}
	from \eqref{a*} and \eqref{b*}, we have that $\psi$ has property \eqref{cond:3}, and \eqref{cond:1}, \eqref{cond:2} trivially follows. Since $\tau_k=\psi\ast \psi\ast... \psi$ ($k$ times), applying Corollary \ref{induction result}, we get our result.

%

\section{Application to the Titchmarsh divisor problem}
\label{sec:Titchmarsh problem}
As an application of our main result for the prime indicator function, we first prove Theorem \ref{Titchmarsh problem} in this section.
\begin{proof}[Proof on Theorem \ref{Titchmarsh problem}]
 For a fixed $a\in \F_{q}[t]$, consider the sum
\begin{equation}
	\label{sum S}
	S:=\sum\limits_{\substack{P_{1}, P_{2} \\ \deg(P_{1}P_{2})=N}}
	\tau(P_{1}P_{2}-a),
\end{equation}
where $P_i$ represents a monic irreducible polynomial in $\F_{q}[t]$. Let $N$ be large enough so that $\deg (a)<N$. 
We split the sum $S$ as follows.

\begin{align*}
	S&=\sum\limits_{\substack{\deg(P_{1}P_{2})=N}}
	\sum\limits_{d|(P_{1}P_{2}-a)} 1\\
	&=\sum\limits_{\substack{\deg(d) \leq N\\(d,a)=1}} \sum\limits_{\substack{\deg(P_{1}P_{2})=N\\ P_{1}P_{2} \equiv a	\pmod*{d}}}
	1+\sum\limits_{\substack{\deg(d) \leq N\\ (d,a)>1}} \sum\limits_{\substack{\deg(P_{1}P_{2})=N\\ P_{1}P_{2} \equiv a	\pmod*{d}}}
	1\\
	&=:S'+S''  \qquad \text{(say)}.
	\end{align*}
Let us first consider the sum $S''$. We can reduce the sum over $d$ as follows.
\begin{align*}
	S''=&\sum\limits_{\substack{\deg(P_{1}P_{2})=N}}
	\sum\limits_{\substack{d|(P_{1}P_{2}-a)\\ \deg(d)\leq \frac{N}{2}, (d,a)>1}}
	1
	+
	\sum\limits_{\substack{\deg(P_{1}P_{2})=N}}
	\sideset{}{'}\sum\limits_{\substack{d|(P_{1}P_{2}-a)\\ \deg(d)> \frac{N}{2},(d,a)>1}}
	1\\
	\leq&2\sum\limits_{\substack{ \deg(d)\leq \frac{N}{2}\\ (d,a)>1}}\sum\limits_{\substack{\deg(P_{1}P_{2})=N\\ P_1P_2\equiv a\pmod*d}}
	1
	\end{align*}
Observe that the inner sum is non trivial only when the gcd $(d,a)$ is a prime. Therefore, this can be bounded using \eqref{PNTforAP} and Lemma \ref{bound for phi} as
\begin{align*}
	S''\leq &2\sum_{P|a}\sum_{\substack{\deg (d)\leq \frac{N}{2}\\(d,a)=P}}~\sum_{\substack{\deg (P_1)=N-\deg (P)\\P_1\equiv \frac{a}{P}\pmod*{\frac{d}{P}}}}1 \nonumber\\
	\ll&\sum_{P|a}\frac{q^{N-\deg (P)}}{N-\deg(P)}\sum_{\deg \left(\frac{d}{P}\right)\leq \frac{N}{2}-\deg (P)}\frac{1}{\phi\left(\frac{d}{P}\right)}\ll q^N.\end{align*}
Therefore, 
\begin{align}
	S=S'+O\left(q^N\right).\label{Bound for S''}
\end{align}
We now consider $S'$ and split the sum in $S'$ as follows.
\begin{align}
S'	=&\sum\limits_{\substack{\deg(P_{1}P_{2})=N}}
	\sideset{}{'}\sum\limits_{\substack{d|(P_{1}P_{2}-a)\\ \deg(d)<\frac{N}{2}}}
	1
	+
	\sum\limits_{\substack{\deg(P_{1}P_{2})=N}}
	\sideset{}{'}\sum\limits_{\substack{d|(P_{1}P_{2}-a)\\ \deg(d)> \frac{N}{2}}}
	1
	+
	\sum\limits_{\substack{\deg(P_{1}P_{2})=N}}
	\sideset{}{'}\sum\limits_{\substack{d|(P_{1}P_{2}-a)\\ \deg(d)= \frac{N}{2}}}1\nonumber\\
	=& 2\sum\limits_{\substack{\deg(P_{1}P_{2})=N}} 
	\sideset{}{'}\sum\limits_{\substack{d|(P_{1}P_{2}-a)\\ \deg(d)<\frac{N}{2}}}
	1
	+
	\sum\limits_{\substack{\deg(P_{1}P_{2})=N}}
	\sideset{}{'}\sum\limits_{\substack{d|(P_{1}P_{2}-a)\\ \deg(d)=\frac{N}{2}}}
	1\nonumber\\
	=:& 2S_{1}+S_{2} \qquad \text{(say)}.\label{S}
\end{align}
	Here and now onwards, $ \sideset{}{'}\sum$ represents that the sum is restricted to $d$ with $(d,a)=1$. Let $Q:=\frac{N}{2}-B\log N$, for some $B>0$ to be chosen later. We split the range of $\deg (d)$ in $S_1$ into two parts, which gives
	\begin{align}
		S_1
		=& \nonumber
		\sum\limits_{\substack{\deg(d) \leq Q\\(d,a)=1}}  
		\sum\limits_{\substack{\deg(P_{1}P_{2})=N\\ P_{1}P_{2} \equiv a	\pmod*{d}}}
		1+	\sum\limits_{\substack{Q<\deg(d)< \frac{N}{2}\\(d,a)=1}}  
		\sum\limits_{\substack{\deg(P_{1}P_{2})=N\\ P_{1}P_{2} \equiv a	\pmod*{d}}}
		1
		\\
		=:&
		\, S_{11}+S_{12} \qquad \text{(say)}. \label{sum in two parts}
	\end{align}
	First, we concentrate on the sum $S_{11}$. We have
	\begin{align*}
	S_{11}=&	 \sideset{}{'}\sum\limits_{\substack{\deg(d) \leq Q}} \sum\limits_{\substack{\deg(P_{1}P_{2})=N\\ P_{1}P_{2} \equiv a	\pmod*{d}}}
		1\\
		=&	 \sideset{}{'}\sum\limits_{\substack{\deg(d) \leq Q}}\sum\limits_{\substack{\deg(f)=N\\
				f \equiv a\pmod* d}} 
		\mathbbm{1}_{\mathcal{P}} \ast \mathbbm{1}_{\mathcal{P}} (f)\\
		=&
		 \sideset{}{'}\sum\limits_{\substack{\deg(d) \leq Q}}\Bigg(\frac{1}{\phi(d)} 
		\sum\limits_{\substack{\deg(f)=N\\
				(f,\,d)=1 }} 
		\mathbbm{1}_{\mathcal{P}} \ast \mathbbm{1}_{\mathcal{P}} (f)
		+
		E(N; d, a; \mathbbm{1}_{\mathcal{P}} \ast \mathbbm{1}_{\mathcal{P} })\Bigg),
	\end{align*}
where $E(N; d, a; \mathbbm{1}_{\mathcal{P}} \ast \mathbbm{1}_{\mathcal{P} })$ is as defined in \eqref{def:error}. Therefore,
	\begin{align*}
		S_{11}= &
			 \sideset{}{'}\sum\limits_{\substack{\deg(d) \leq Q}}  
		\frac{1}{\phi(d)} 
		\sum\limits_{\substack{\deg(P_{1}P_{2})=N\\ (P_{1}P_{2},\,d)=1 }}
		1
		+
			 \sideset{}{'}\sum\limits_{\substack{\deg(d) \leq Q}}
		E(N; d, a; \mathbbm{1}_{\mathcal{P}} \ast \mathbbm{1}_{\mathcal{P} }).
	\end{align*} 
	Let $A=1$ in Corollary \ref{prime indicator function} and choose $B=B(1)$. Thus we have
	\begin{align*}
		 \sideset{}{'}\sum\limits_{\substack{\deg(d) \leq Q}}  
		E(N; d, 1; \mathbbm{1}_{\mathcal{P}} \ast \mathbbm{1}_{\mathcal{P} })
		=&
		O\left( 
		\frac{q^{N}}{N}
		\right), 
	\end{align*}
	Using this bound we further write the sum as
	\begin{align*}
		S_{11}
		=&
		S_{11}^1 -S_{11}^2 + O\left( 
		\frac{q^{N}}{N}
		\right),  
	\end{align*}
	where 
	\begin{align*}
		S_{11}^1
		:=&
			 \sideset{}{'}\sum\limits_{\substack{\deg(d) \leq Q}}  
		\frac{1}{\phi(d)} 
		\sum\limits_{\substack{\deg(P_{1}P_{2})=N}}
		1
		\qquad
		\text{and} \qquad
		S_{11}^2
		:=&	 \sideset{}{'}\sum\limits_{\substack{\deg(d) \leq Q}} 
		\frac{1}{\phi(d)} 
		\sum\limits_{\substack{\deg(P_{1}P_{2})=N\\ (P_{1}P_{2},\, d)\neq 1 }}
		1.
	\end{align*}
	The main contribution to our sum comes from $S_{11}^1$. We determine this as follows.
	\begin{align}
		S_{11}^1
		=& \sideset{}{'}\sum\limits_{\substack{\deg(d) \leq Q}}	\frac{1}{\phi(d)}\sum\limits_{k=1}^{N-1}
		\sum\limits_{ \deg(P_{1})=k} 
		1
		\sum\limits_{ \deg(P_{2})=N-k} 1.
				\label{Bound for S11}
	\end{align}
	Using \eqref{PNT} to compute the inner sum, we get,
	\begin{align}
		\sum\limits_{k=1}^{N-1}
		\sum\limits_{ \deg(P_{1})=k} 
		1
		\sum\limits_{ \deg(P_{2})=N-k} 1=&\sum\limits_{k>\log N}^{N-\log N}
		\sum\limits_{ \deg(P_{1})=N-k} 
		1
		\sum\limits_{ \deg(P_{2})=k}1+		
	2\sum\limits_{k=1}^{\log N}
		\sum\limits_{ \deg(P_{1})=N-k} 
		1
		\sum\limits_{ \deg(P_{2})=k} 
		1
		\nonumber\\
		=&
		\sum\limits_{k>\log N}^{N-\log N}
		\frac{q^{k}}{k}
		\left(
		1+O\left(  q^{-\frac{k}{2}}   \right) 
		\right) 
		\frac{q^{N-k}}{N-k}
		\left(
		1+O\left(  q^{-\frac{N-k}{2}}   \right) 
		\right)\nonumber\\
		+&O\left(\frac{q^{N}}{N}\log\log N\right) \nonumber\\
		=&\frac{q^{N}}{N}
		\sum\limits_{k>\log N}^{N-\log N}
		\left( 
		\frac{1}{k}
		+
		\frac{1}{N-k}
		\right)
		\left(
		1+O\left(  q^{-\frac{k}{2}}   \right) 
		+ O\left(  q^{-\frac{N-k}{2}}   \right)
		\right)\nonumber\\
		+&O\left(\frac{q^{N}}{N}\log\log N\right)\nonumber\\
		=&\frac{2q^{N}}{N}
		\left( 
		\log N+O\left(\log\log N\right)
		\right) 
		\left( 
		1+ O\left( q^{-\frac{\log N}{2}}   \right) 
		\right)+O\left(\frac{q^{N}}{N}\log\log N\right) \nonumber\\
		=&
		\frac{2q^{N}}{N}\log N
		+
		O\left( 
		\frac{q^{N}\log N}{N^{\frac{3}{2}}}
		\right)+O\left(\frac{q^N\log\log N}{N}\right).\label{Bound for S11-1}
	\end{align}
	Substituting \eqref{Bound for S11-1} into \eqref{Bound for S11} and using Lemma \ref{bound for phi} for the fixed polynomial $a$, we obtain
	\begin{align*}
		S_{11}^1
		=&
		\left( C_a\frac{\zeta_q(2)\zeta_q(3)}{\zeta_q(6)}Q+O(1)\right)
		\left( 
		\frac{2q^{N}}{N}\log N
		+
		O\left( 
		\frac{q^{N}\log N}{N^{\frac{3}{2}}}
		\right)+O\left(\frac{q^N\log\log N}{N}\right)
		\right) \\
		=&
		2C_a\frac{\zeta_q(2)\zeta_q(3)}{\zeta_q(6)}\frac{q^{N}}{N} Q(\log N)
	+O\left(q^{N} \log\log N \right),
	\end{align*}
where $C_a$ is as in \eqref{def: Ca}. Now, putting $Q=\frac{N}{2}-B\log N$, we have
	\begin{align}
		S_{11}^1
		=&
		C_a\frac{\zeta_q(2)\zeta_q(3)}{\zeta_q(6)}q^{N}\log N 
	+O\left(q^{N} \log\log N \right).\label{S11^1}
	\end{align}
	We will now show that the contribution from $S_{11}^{2}$ is negligible. If $(P_{1}P_{2},\, d)\neq 1$, we have $P_{1}|d$ or $P_{2}|d$. Since both the cases are symmetric, we have
\begin{align*}
	S_{11}^2
	\ll&
	\sum\limits_{\deg(d) \leq Q}  
	\frac{1}{\phi(d)} 
	\sum\limits_{\substack{\deg(P_{1}P_{2})=N\\ P_{1}|d }}
	1.
\end{align*}
By writing $d=P_{1}^{m}k$ such that $(P_{1}, k)=1$, we see that
\begin{align*}
	S_{11}^2
	\ll &
	\sum\limits_{ \deg(P_{1}P_{2})=N}  
	\sum\limits_{m\geq 1}
	\frac{1}{\phi(P_{1}^{m})} 
	\sum\limits_{\deg(k)\leq Q}
	\frac{1}{\phi(k)}.
\end{align*}
Further, using Lemma \ref{bound for phi} with $g=1$, the innermost sum is $\ll Q$. Using the bound $\sum\limits_{m\geq 1}
\frac{1}{\phi(P_{1}^{m})}\ll\frac{1}{|P_1|}$ and \eqref{PNT}, we have
\begin{align}
	S_{11}^{2}\ll&
	2Q\sum\limits_{k=1}^{\log N}
	\sum\limits_{ \deg(P_{1})=k} 
	\frac{1}{|P_{1}|}
	\sum\limits_{ \deg(P_{2})=N-k} 
	1+Q\sum\limits_{k>\log N}^{N-\log N}
	\sum\limits_{ \deg(P_{1})=N-k} 
	\frac{1}{|P_{1}|}
	\sum\limits_{ \deg(P_{2})=k}
	1\nonumber\\
	\ll& N \log\log N+\frac{q^{N}}{N}\log N\ll \frac{q^{N}}{N}\log N\label{S11^2}.
\end{align}
	Therefore, from \eqref{S11^1} and \eqref{S11^2}, we have
	\begin{align}
		S_{11}
		=&
		C_a\frac{\zeta_q(2)\zeta_q(3)}{\zeta_q(6)}q^{N}\log N
		+O\left(q^{N} \log\log N \right).\label{S11}
	\end{align}
	Next, we estimate the sum $S_{12}$ by using the analogue of the Brun-Titchmarsh inequality over $\F_{q}[t]$ due to Hsu as stated in  Lemma \ref{Brun-Titchmarsh}. Recall that
	\begin{align*}
		S_{12}=	\sideset{}{'} \sum\limits_{Q<\deg(d)< \frac{N}{2}}  
		\sum\limits_{\substack{\deg(P_{1}P_{2})=N\\ P_{1}P_{2} \equiv a	\pmod*{d}}}
		1.
	\end{align*}
	From Lemma \ref{Brun-Titchmarsh} and \eqref{PNT}, we see that
	\begin{align}
		S_{12}=	& \sideset{}{'}\sum\limits_{\substack{Q<\deg(d)< \frac{N}{2}}}\left(\sum\limits_{k\leq \frac{N}{2}}
		\sum\limits_{ \substack{\deg(P_{1})=k\\ (P_1,d)=1}} 
		\sum\limits_{\substack{\deg(P_{2})=N-k\\ P_2\equiv aP_1^{-1}\pmod*{d}}} 
		1+\sum\limits_{k<\frac{N}{2}}
		\sum\limits_{ \substack{\deg(P_{2})=k\\ (P_2,d)=1}} 
		\sum\limits_{\substack{\deg(P_{1})=N-k\\ P_1\equiv aP_2^{-1}\pmod*{d}}} 
		1\right)\nonumber\\
		\ll&\sum\limits_{\substack{Q<\deg(d)< \frac{N}{2}}}\sum\limits_{k\leq \frac{N}{2}}
		\left(\frac{q^{k}}{k}\right)\left(\frac{q^{N-k}}{\phi(d)(N-k+\deg d+1)}\right)\ll_B \frac{q^N}{N}(\log N)^2,\label{S12}
	\end{align}
	where $aP_i^{-1}$ is the residue class modulo $d$ such that $P_iP_i^{-1}\equiv a\pmod*{d}$. Here we have used Lemma \ref{bound for phi} to write the final bound. Using \eqref{S11} and \eqref{S12}, we have from \eqref{sum in two parts} that
		\begin{align}
		S_{1}
		=&
	C_a	\frac{\zeta_q(2)\zeta_q(3)}{\zeta_q(6)}q^{N}\log N
		+
		O\left( 
		q^{N}
		\log\log N
		\right).\label{S1}
	\end{align}
	It remains to estimate $S_2$. Again we split the sum as 
	\begin{align*}
		S_2=&2\sum_{\substack{\deg d=\frac{N}{2}\\(d,a)=1}} \sum\limits_{k<\frac{N}{2}}
		\sum\limits_{ \substack{\deg(P_{1})=k\\ (P_1,d)=1}} 
		\sum\limits_{\substack{\deg(P_{2})=N-k\\ P_2\equiv aP_1^{-1}\pmod*{d}}} 1
		+\sum_{\substack{\deg d=\frac{N}{2}\\(d,a)=1}}\sum\limits_{ \substack{\deg(P_{2})=\frac{N}{2}\\ (P_2,d)=1}} 
		\sum\limits_{\substack{\deg(P_{1})=\frac{N}{2}\\ P_1\equiv aP_2^{-1}\pmod*{d}}} 
		1.
	\end{align*}
	The first term in the above expression is $ \ll_Bq^N(\log N)^2/N$ as done in \eqref{S12}. For the second term, observe that the arithmetic progression $aP_2^{-1}\pmod*{d}$ can contain atmost one prime of degree $\frac{N}{2}$. Indeed, for a general term $aP_2^{-1}+fd$ in the arithmetic progression, $f$ must have degree zero as $\deg (d)=\frac{N}{2}$. Since we are counting monics, $f$ must be $1$. Thus, using \eqref{PNT},we have that the second term above is $\ll q^N/N$. Thus, we have 
	\begin{align}
		S_2\ll_B\frac{q^N}{N}(\log N)^2.\label{S2}
	\end{align}
	Using \eqref{S1} and \eqref{S2} in \eqref{S} and \eqref{Bound for S''} completes the proof of Theorem \ref{Titchmarsh problem}.
\end{proof}
	
	\begin{proof}[Proof of Theorem \ref{q>inf version}]
		Proof follows by taking $q\rightarrow\infty$ Theorem \ref{Titchmarsh problem} with slight modification. Note that for $q\rightarrow\infty$, we have $S"\ll q^{N-1}$. Also, from equation \eqref{Bound for S11-1} we have
		\begin{align*}
			\sum_{k=1}^{N-1}\sum_{\deg P_1=k}1\sum_{\deg P_2=N-k}1&=\sum_{k=1}^{N-1}\frac{q^k}{k}\left(1+O\left(q^{-k/2}\right)\right)\frac{q^{N-k}}{N-k}\left(1+O\left(q^{-(N-k)/2}\right)\right)\\
			&=\frac{2q^N}{N}\sum_{k=1}^{N-1}\frac{1}{k}\left(1+O\left(q^{-1/2}\right)\right)\\
			&=\frac{2q^N}{N}(\log N+\gamma)+O\left(\frac{q^N}{N^2}\right)+O\left(\frac{q^{N-\frac{1}{2}}}{N}\log N\right),
		\end{align*}
		as $q\rightarrow\infty$. Therefore,
		\begin{align}
			S_{11}^1=q^N(\log N +\gamma)+O\left(\frac{q^N}{N}(\log N)^2\right)+O\left(q^{N-\frac{1}{2}}\log N\right),
		\end{align}
		as the constant in Lemma \ref{bound for phi} tends to $0$ as $q\rightarrow \infty$. Considering $q\rightarrow\infty$ in rest of the argument, we derive the result.
	\end{proof}

	\noindent
	{\textbf{Acknowledgments.}} 
	Both the authors express their sincere gratitude to Prof. Akshaa Vatwani for suggesting the problem that led to this work, as well as for her insightful discussions and suggestions on the relevant references. The second author is very thankful to Prof. M. Ram Murty for his lectures on ``Arithmetic in function fields"	at IIT Gandhinagar which were instrumental in motivating this project. The funding from the MHRD SPARC project SPARC/2018 -2019/P567/SL, under which these lectures were organized is gratefully acknowledged.
	
	\bibliographystyle{amsplain}
\bibliography{Induction_principle_for_BV.bib}	

\providecommand{\bysame}{\leavevmode\hbox to3em{\hrulefill}\thinspace}
\providecommand{\MR}{\relax\ifhmode\unskip\space\fi MR }
\providecommand{\MRhref}[2]{%
  \href{http://www.ams.org/mathscinet-getitem?mr=#1}{#2}
}
\providecommand{\href}[2]{#2}
\begin{thebibliography}{10}

\bibitem{AkGh}
A.~Akbary and D.~Ghioca, \emph{A geometric variant of {T}itchmarsh divisor
  problem}, Int. J. Number Theory \textbf{8} (2012), no.~1, 53--69.
  \MR{2887882}

\bibitem{Ru}
J.~C. Andrade, L.~Bary-Soroker, and Z.~Rudnick, \emph{Shifted convolution and
  the {T}itchmarsh divisor problem over {$\mathbb F_q[t]$}}, Philos. Trans.
  Roy. Soc. A \textbf{373} (2015), no.~2040, 20140308, 18. \MR{3338116}

\bibitem{BaSi}
S.~Baier and R.~K. Singh, \emph{Large sieve inequality with power moduli for
  function fields}, J. Number Theory \textbf{196} (2019), 1--13. \MR{3906465}

\bibitem{BaSiErratum}
\bysame, \emph{Erratum to ``{L}arge sieve inequality with power moduli for
  function fields'' [{J}. {N}umber {T}heory 196 (2019) 1--13]}, J. Number
  Theory \textbf{210} (2020), 431--432. \MR{4057535}

\bibitem{BaSi2}
\bysame, \emph{The large sieve inequality with square moduli for quadratic
  extensions of function fields}, Int. J. Number Theory \textbf{16} (2020),
  no.~9, 1907--1922. \MR{4153360}

\bibitem{primesinAP_BFI}
E.~Bombieri, J.~B. Friedlander, and H.~Iwaniec, \emph{Primes in arithmetic
  progressions to large moduli}, Acta Math. \textbf{156} (1986), no.~3-4,
  203--251. \MR{834613}

\bibitem{primesinAPII_BFI}
\bysame, \emph{Primes in arithmetic progressions to large moduli. {II}}, Math.
  Ann. \textbf{277} (1987), no.~3, 361--393. \MR{891581}

\bibitem{primesinAPIII_BFI}
\bysame, \emph{Primes in arithmetic progressions to large moduli. {III}}, J.
  Amer. Math. Soc. \textbf{2} (1989), no.~2, 215--224. \MR{976723}

\bibitem{RmAc}
A.~C. Cojocaru and M.~R. Murty, \emph{An introduction to sieve methods and
  their applications}, London Mathematical Society Student Texts, vol.~66,
  Cambridge University Press, Cambridge, 2006. \MR{2200366}

\bibitem{DaMu}
P.~Darbar and A.~Mukhopadhyay, \emph{A {B}ombieri-type theorem for convolution
  with application on number field}, Acta Math. Hungar. \textbf{163} (2021),
  no.~1, 37--61. \MR{4217957}

\bibitem{Drappeau}
S.~Drappeau, \emph{Sums of {K}loosterman sums in arithmetic progressions, and
  the error term in the dispersion method}, Proc. Lond. Math. Soc. (3)
  \textbf{114} (2017), no.~4, 684--732. \MR{3653244}

\bibitem{Comident}
S.~Drappeau and B.~Topacogullari, \emph{Combinatorial identities and
  {T}itchmarsh's divisor problem for multiplicative functions}, Algebra Number
  Theory \textbf{13} (2019), no.~10, 2383--2425. \MR{4047638}

\bibitem{Fe}
A.~T. Felix, \emph{Generalizing the {T}itchmarsh divisor problem}, Int. J.
  Number Theory \textbf{8} (2012), no.~3, 613--629. \MR{2904920}

\bibitem{Fo}
\'{E}. Fouvry, \emph{Sur le probl\`eme des diviseurs de {T}itchmarsh}, J. Reine
  Angew. Math. \textbf{357} (1985), 51--76. \MR{783533}

\bibitem{FuIw}
\'{E.} Fouvry and H.~Iwaniec, \emph{On a theorem of {B}ombieri-{V}inogradov
  type}, Mathematika \textbf{27} (1980), no.~2, 135--152 (1981). \MR{610700}

\bibitem{FoIwII}
\bysame, \emph{Primes in arithmetic progressions}, Acta Arith. \textbf{42}
  (1983), no.~2, 197--218. \MR{719249}

\bibitem{Fu}
A.~Fujii, \emph{A local study of some additive problems in the theory of
  numbers}, Proc. Japan Acad. \textbf{52} (1976), no.~3, 113--115. \MR{399021}

\bibitem{Ga}
P.~X. Gallagher, \emph{The large sieve}, Mathematika \textbf{14} (1967),
  14--20. \MR{214562}

\bibitem{GAS}
A.~Granville and X.~Shao, \emph{Bombieri-{V}inogradov for multiplicative
  functions, and beyond the {$x^{1/2}$}-barrier}, Adv. Math. \textbf{350}
  (2019), 304--358. \MR{3947647}

\bibitem{Ha}
H.~Halberstam, \emph{Footnote to the {T}itchmarsh-{L}innik divisor problem},
  Proc. Amer. Math. Soc. \textbf{18} (1967), 187--188. \MR{204379}

\bibitem{Hi2}
J.~G. Hinz, \emph{Methoden des grossen {S}iebes in algebraischen
  {Z}ahlk\"{o}rpern}, Manuscripta Math. \textbf{57} (1987), no.~2, 181--194.
  \MR{871630}

\bibitem{Hi1}
\bysame, \emph{A generalization of {B}ombieri's prime number theorem to
  algebraic number fields}, Acta Arith. \textbf{51} (1988), no.~2, 173--193.
  \MR{975109}

\bibitem{Hsu-Largesieve}
C.~Hsu, \emph{A large sieve inequality for rational function fields}, J. Number
  Theory \textbf{58} (1996), no.~2, 267--287. \MR{1393616}

\bibitem{Hsu-BrunTitchmarsh}
\bysame, \emph{Applications of the large sieve inequality for {$\mathbf
  F_q[T]$}}, Finite Fields Appl. \textbf{4} (1998), no.~3, 275--281.
  \MR{1640777}

\bibitem{Hu1}
M.~N. Huxley, \emph{The large sieve inequality for algebraic number fields},
  Mathematika \textbf{15} (1968), 178--187. \MR{237455}

\bibitem{Hu2}
\bysame, \emph{The large sieve inequality for algebraic number fields. {III}.
  {Z}ero-density results}, J. London Math. Soc. (2) \textbf{3} (1971),
  233--240. \MR{276196}

\bibitem{Jo}
J.~Johnsen, \emph{On the large sieve method in {${\rm GF}[q,\,x]$}},
  Mathematika \textbf{18} (1971), 172--184. \MR{302617}

\bibitem{Oleksiy}
O.~Klurman, A.~P. Mangerel, and J.~Ter{\"a}v{\"a}inen, \emph{Correlations of
  multiplicative functions in function fields}, arXiv preprint arXiv:2009.13497
  (2020).

\bibitem{Li}
U.~V. Linnik, \emph{{T}he large sieve.}, C. R. (Doklady) Acad. Sci. URSS (N.S.)
  \textbf{30} (1941), 292--294. \MR{0004266}

\bibitem{Li-book}
\bysame, \emph{The dispersion method in binary additive problems}, American
  Mathematical Society, Providence, R.I., 1963, Translated by S. Schuur.
  \MR{0168543}

\bibitem{maynard2020primes}
J.~Maynard, \emph{Primes in arithmetic progressions to large moduli {I}: fixed
  residue classes}, arXiv preprint arXiv:2006.06572 (2020).

\bibitem{maynard2020primesII}
\bysame, \emph{Primes in arithmetic progressions to large moduli {II}:
  Well-factorable estimates}, arXiv preprint arXiv:2006.07088 (2020).

\bibitem{maynard2020primesIII}
\bysame, \emph{Primes in arithmetic progressions to large moduli {III}: Uniform
  residue classes}, arXiv preprint arXiv:2006.08250 (2020).

\bibitem{Mo}
Y.~Motohashi, \emph{An induction principle for the generalization of
  {B}ombieri's prime number theorem}, Proc. Japan Acad. \textbf{52} (1976),
  no.~6, 273--275. \MR{422179}

\bibitem{Rod}
G.~Rodriquez, \emph{Sul problema dei divisori di {T}itchmarsh}, Boll. Un. Mat.
  Ital. (3) \textbf{20} (1965), 358--366. \MR{0197409}

\bibitem{Ro}
M.~Rosen, \emph{Number theory in function fields}, Graduate Texts in
  Mathematics, vol. 210, Springer-Verlag, New York, 2002. \MR{1876657}

\bibitem{sawin_Duke}
W.~Sawin, \emph{Square-root cancellation for sums of factorization functions
  over short intervals in function fields}, Duke Mathematical Journal
  \textbf{170} (2021), no.~5, 997--1026.

\bibitem{sawin2021square}
\bysame, \emph{Square-root cancellation for sums of factorization functions
  over squarefree progressions in $\mathbb{F}_q [t] $}, arXiv preprint
  arXiv:2102.09730 (2021).

\bibitem{sawin2018chowla}
W.~Sawin and M.~Shusterman, \emph{On the {C}howla and twin primes conjectures
  over {$\mathbb{F}_q [T] $}}, arXiv preprint arXiv:1808.04001 (2018).

\bibitem{Tenenbaum}
G.~Tenenbaum, \emph{Introduction to analytic and probabilistic number theory},
  third ed., Graduate Studies in Mathematics, vol. 163, American Mathematical
  Society, Providence, RI, 2015, Translated from the 2008 French edition by
  Patrick D. F. Ion. \MR{3363366}

\bibitem{Ti}
C.~E. Titchmarsh, \emph{A divisor problem}, Rend. Circ. Mat. Palermo
  \textbf{54} (1930), 414--429.

\bibitem{VaPe}
A.~Vatwani and P.~Wong, \emph{On generalizations of the {T}itchmarsh divisor
  problem}, Acta Arith. \textbf{193} (2020), no.~4, 321--337. \MR{4074394}

\bibitem{Zhang}
Y.~Zhang, \emph{Bounded gaps between primes}, Ann. of Math. (2) \textbf{179}
  (2014), no.~3, 1121--1174. \MR{3171761}

\end{thebibliography}
\end{document}